\documentclass[12pt]{amsart}
\usepackage[top=1in, bottom=1in, left=1in, right=1in]{geometry}
\usepackage{amsfonts}
\usepackage{amsmath}
\usepackage{comment}
\usepackage{amssymb}
\usepackage{bbm}
\usepackage{comment}
\usepackage{mathrsfs}
\numberwithin{equation}{section}
\usepackage{times}
\usepackage[backend=biber, sorting=nyt, url=false,
maxnames = 100,
doi=false]{biblatex}
\usepackage{siunitx}
\usepackage[usenames,dvipsnames]{color}
\usepackage{comment}
\usepackage{xcolor}
\usepackage{mathtools}
\usepackage{bm}
\usepackage{esvect}
\usepackage{hyperref}
\hypersetup{
    colorlinks = true,
linkcolor={black},
urlcolor={blue},
citecolor={blue},    
urlcolor = {blue},
citebordercolor = {0.33 .58 0.33},
 linkbordercolor = {0.99 .28 0.23},
 breaklinks=true}
 
\newcommand{\kommentar}[1]{}

\newcommand{\R}{\mathbb{R}}

\newcommand{\Z}{\mathbb{Z}}

\newtheorem{thm}{Theorem}[section]

\newtheorem{coro}[thm]{Corollary}
\newtheorem{lem}[thm]{Lemma}
\newtheorem{rem}[thm]{Remark}
\newtheorem{conj}[thm]{Conjecture}

\newcommand{\dd}{\;\mathrm{d}}

\theoremstyle{remark}

\usepackage{geometry}
\DeclarePairedDelimiter\floor{\lfloor}{\rfloor}

\title[Representations as Sums of Icosahedral and Dodecahedral
Numbers]{Representations as Sums of Icosahedral and Dodecahedral
Numbers: Proof of Pollock's Conjectures}

\author{Debmalya Basak, Anji Dong, Katerina Saettone and Alexandru Zaharescu}

\address{
Debmalya Basak: Department of Mathematics,
University of Illinois Urbana-Champaign,
Altgeld Hall, 1409 W. Green Street,
Urbana, IL, 61801, USA}
\email{dbasak2@illinois.edu}

\address{
Anji Dong: Department of Mathematics,
University of Illinois Urbana-Champaign,
Altgeld Hall, 1409 W. Green Street,
Urbana, IL, 61801, USA}
\email{anjid2@illinois.edu}

\address{
Katerina Saettone: Department of Mathematics,
University of Illinois Urbana-Champaign,
Altgeld Hall, 1409 W. Green Street,
Urbana, IL, 61801, USA}
\email{kas18@illinois.edu}

\address{
Alexandru Zaharescu: Department of Mathematics,
University of Illinois Urbana-Champaign,
Altgeld Hall, 1409 W. Green Street,
Urbana, IL, 61801, USA and Simion Stoilow Institute of Mathematics of the Romanian Academy, 
P. O. Box 1-764, RO-014700 Bucharest, Romania}
\email{zaharesc@illinois.edu}  
\begin{document}
\nocite{*}
\setcounter{tocdepth}{1}
\maketitle
{\centering\footnotesize \textit{We dedicate the paper to George Andrews and Bruce Berndt on the occasion of their 85\textsuperscript{th} birthdays.}\par}

\begin{abstract}
On the occasion of George Andrews' and Bruce Berndt's combined 170\textsuperscript{th} birthday, we prove two conjectures of Sir Frederick Pollock that are more than 170 years old. In 1843, Pollock conjectured that every positive integer is the sum of at most 13 icosahedral numbers, or at most 21 dodecahedral numbers. We establish refined versions of both conjectures. 
\end{abstract}
\setcounter{tocdepth}{-1}
\section{Introduction}\label{sec: Introduction}
In 1843, Sir Frederick Pollock \cite{pollock} stated a number of conjectures that are extensions of the Fermat polygonal number theorem to other figurate numbers. Among other things, he proposed five conjectures on polyhedral numbers associated with the five Platonic solids. To be precise, he conjectured that every positive integer is the sum of at most 5 tetrahedral numbers, 7 octahedral numbers, 9 cubes, 13 icosahedral numbers, or 21 dodecahedral numbers. In the present paper, we solve Pollock's conjectures for icosahedral and dodecahedral numbers. Before turning to the historical background of Pollock’s conjectures and related topics, we begin by presenting our results concerning representations as sums of icosahedral and dodecahedral numbers. For $n \in \mathbb{N}$, let $I_n$ and $D_n$ denote the $n$-th icosahedral and dodecahedral numbers respectively, defined by
\begin{align}\label{Icosahedral Number : defn}
I_n &= \frac{n(5n^2-5n+2)}{2},\\
\textrm{and} \quad D_n &= \frac{n(3n-1)(3n-2)}{2}.\label{Dodecahedral Number : defn}
\end{align}

\begin{thm}\label{Theorem: Representation Icosahedral}
For any positive integer $m$, let $\mathcal{I}_9(m)$ denote the number of representations of $m$ as a sum of 9 icosahedral numbers. Then for all $m>e^{e^{94}}$,
\begin{align*}
\bigg|\mathcal{I}_9(m) -\frac{4}{125}\bigg(\Gamma\bigg(\frac{4}{3}\bigg)\bigg)^{9}\mathfrak{S}_{9,\mathcal{I}}(m) m^{2}\bigg|\leqslant 10^{12}m^{2-\frac{1}{108}},
\end{align*}
where $\Gamma(s)$ is the Gamma function, and the arithmetic factor $\mathfrak{S}_{9,\mathcal{I}}(m)$, defined by \eqref{Arithmetic Factor for s=9} below, satisfies the inequality
\begin{align*}
\mathfrak{S}_{9,\mathcal{I}}(m)&\geqslant \frac{1}{e^{ e^{92}}}.  
\end{align*}
\end{thm} 
\begin{thm}\label{Theorem: Representation Dodecahedral}
For any positive integer $m$, let $\mathcal{D}_9(m)$ denote the number of representations of $m$ as a sum of 9 dodecahedral numbers. Then for all $m> e^{e^{94}}$, 
\begin{align*}
\bigg|\mathcal{D}_{9}(m) -\frac{4}{729}\bigg(\Gamma\bigg(\frac{4}{3}\bigg)\bigg)^{9}\mathfrak{S}_{9,\mathcal{D}}(m)m^{2}\bigg|\leqslant 
     10^{16}m^{2-\frac{1}{108}},
\end{align*}
where $\Gamma(s)$ is the Gamma function, and the arithmetic factor $\mathfrak{S}_{9,\mathcal{D}}(m)$, defined similarly as \eqref{Arithmetic Factor for s=9}, satisfies the inequality
\begin{align*}
     \mathfrak{S}_{9,\mathcal{D}}(m) \geqslant\frac{1}{e^{e^{92}} } .
\end{align*}
\end{thm}

To put Theorems \ref{Theorem: Representation Icosahedral} and \ref{Theorem: Representation Dodecahedral} into context, we now state Pollock's conjectures. 

\begin{conj}\label{Pollock's Conjecture Icosahedral}
Every positive integer is the sum of at most 13 icosahedral numbers.
\end{conj}
\begin{conj}\label{Pollock's Conjecture Dodecahedral}
Every positive integer is the sum of at most 21 dodecahedral numbers.
\end{conj}

However, Conjectures \ref{Pollock's Conjecture Icosahedral} and \ref{Pollock's Conjecture Dodecahedral}, as stated, are false. One can check that the number 95 can be written as a sum of 15 icosahedral numbers, 
\begin{align}\label{95}
95=48+12+12+12+1+1+1+1+1+1+1+1+1+1+1,
\end{align}
but cannot be written as a sum of at most 14 icosahedral numbers. Meanwhile, the number $79$ can be expressed as the sum of 22 dodecahedral numbers,
\begin{align}
    79=20+20+20+1+1+\cdots+1+1,\label{79}
\end{align}
where there are 19 occurrences of the number $1$ in the sum. But $79$ cannot be written as a sum of at most 21 dodecahedral numbers. Accordingly, a slight refinement of Conjectures \ref{Pollock's Conjecture Icosahedral} and \ref{Pollock's Conjecture Dodecahedral} is required before they can be rigorously proven. We carry out this refinement in our work and present the strongest possible results in this direction. As a first step, building upon Linnik’s method, we establish the following theorems concerning representations as sums of eight icosahedral and dodecahedral numbers, respectively.

\begin{thm}\label{Theorem: 8 Icosahedral Numbers}
    Any $m \in \mathbb{N}$ with $m\geqslant 10^{36}$ is a sum of 8 icosahedral numbers.
\end{thm}
\begin{thm}\label{Theorem: 8 Dodecahedral Numbers}
Any $m \in \mathbb{N}$ with $m\geqslant 10^{39}$ is a sum of 8 dodecahedral numbers.
\end{thm}

Finally, we prove corrected versions of Conjectures \ref{Pollock's Conjecture Icosahedral} and \ref{Pollock's Conjecture Dodecahedral}, which follow as corollaries of Theorems \ref{Theorem: 8 Icosahedral Numbers} and \ref{Theorem: 8 Dodecahedral Numbers}, along with some additional reductions and computational techniques.

\begin{thm}\label{Theorem: Pollock Icosahedral}
Any positive integer can be written as a sum of at most 15 icosahedral numbers. Moreover, 15 is the smallest positive integer with this property. 
\end{thm}
\begin{thm}\label{Theorem: Pollock Dodecahedral}
Any positive integer can be written as a sum of at most 22 dodecahedral numbers.  Moreover, 22 is the smallest positive integer with this property.
\end{thm}

In fact, the only counterexamples to Pollock's original conjecture for icosahedral numbers are the following. The numbers 47, 83, 94, and 119 need 14 terms in their representations as sums of icosahedral numbers, and the number 95 requires 15 terms as seen in \eqref{95}. The only counterexample to Pollock's original conjecture for dodecahedral numbers is 79, shown in \eqref{79}, which requires 22 terms in its representations as sums of dodecahedral numbers.

Before proceeding to the proofs, we give a brief survey of previous results and related methods to put our results into context. We first mention that Pollock's conjecture for cubes, already proposed by Waring in 1770, was confirmed by Wieferich \cite{Wieferich1908BeweisDS} and Kempner \cite{Kempner1912BemerkungenZW} independently between 1909 and 1912. More recently, in 2016, Brady \cite{Brady_2015} proved Pollock's octahedral number conjecture for all integers greater than $e^{10^7}$. In the direction of Pollock's conjectures on icosahedral and dodecahedral numbers, Theorems \ref{Theorem: Representation Icosahedral} and \ref{Theorem: Representation Dodecahedral} establish asymptotic formulas for the number of ways a positive integer $m$ can be represented as sums of icosahedral and dodecahedral numbers respectively, accompanied by explicit power-saving error terms. We put particular emphasis on minimizing the bound for $m$ above for which our asymptotics remain valid. To prove our results, we rely on a combination of analytic techniques, algebraic approaches, and computational methods. We successfully achieve the desired saving, and moreover do it in a concrete, explicit way as required, in Section \ref{sec: Major Arcs II} below. We actually show that the arithmetic factor is uniformly bounded below by $1/e^{e^{92}}$. A similar approach works for dodecahedral numbers. Hence, we only state the result in Theorem \ref{Theorem: Representation Dodecahedral}, with full details to be provided in one of the authors' theses. 

In Section \ref{sec: Proof of Main Theorem}, we establish more general and sharper versions of Theorem \ref{Theorem: Representation Icosahedral} by considering sums of \(s\) icosahedral numbers for any \(s \geqslant 9\). Unfortunately, Theorems \ref{Theorem: Representation Icosahedral} and \ref{Theorem: Representation Dodecahedral}, along with their general versions, do not help us in fully proving Pollock's conjectures. Due to current computational limitations, it is impossible to verify all values of \(m\) up to the lower bound thresholds set by Theorems \ref{Theorem: Representation Icosahedral} and \ref{Theorem: Representation Dodecahedral}. Consequently, to establish the conjectures in their entirety, we turn to alternative approaches to bridge this gap.

Historically, other methods have provided effective bounds for related problems. Hua \cite{hua1936, huagenwaring,hua1940} showed that any sufficiently large integer can be expressed as a sum of at most 8 cubic polynomials of the form 
\[
f(x) = \frac{1}{6}a(x^3-x) + \frac{1}{2}b(x^2-x) + cx + d,
\]
where \(a\), \(b\), \(c\), and \(d\) are integers with \((a,b,c)=1\) and \(a > 0\). However, Hua’s method leads to bounds for the arithmetic factor similar to those in Theorems \ref{Theorem: Representation Icosahedral} and \ref{Theorem: Representation Dodecahedral}, resulting in thresholds as large as \(e^{e^{94}}\). On the other hand, Linnik’s method \cite{linnik1943representation} of expressing all large numbers as sums of seven cubes has led to the complete solution of that problem (see Maillet \cite{maillet1895decomposition}, Watson \cite{watsoncube}, Romani \cite{romani1982computations}, McCurley \cite{mccurley1984effective}, Deshouillers--Hennecart--Landreau \cite{deshouillers20007373170279850}, Ramaré \cite{ramare2005explicit, ramar2007explicit}, Boklan--Elkies \cite{boklan2009every}, Elkies \cite{elkies2010every}, Siksek \cite{siksek2016every}). But in other contexts involving cubic polynomials, the method still leaves bounds that are too large for complete numerical verification. For example, as mentioned above, in the case of Pollock's octahedral number conjecture, Brady \cite{Brady_2015} solved the problem for all numbers larger than the bound $e^{10^{7}}$, which is well beyond current computational limits.

Keeping in mind the above, we shift our attention to sums involving eight terms. To this end, we introduce a novel approach. We start with Linnik's method for sums of seven terms and, in the process, combine it with new algebraic techniques. We make efficient use of the Hasse--Weil bound for the number of points on an elliptic curve modulo a prime $p$. A crucial component of our approach is Bombieri's theorem on exponential sums along curves. This is achieved via a theorem of Cobeli and the fourth author \cite{cobeli} on points on curves satisfying additional congruence constraints. This allows us to establish Theorems \ref{Theorem: 8 Icosahedral Numbers} and \ref{Theorem: 8 Dodecahedral Numbers}, where we successfully reduce the lower bound threshold for $m$ to a level within computational reach. Finally, we verify the conjectures computationally up to the lower bound thresholds and thereby establish Pollock’s conjectures, in particular, proving Theorems \ref{Theorem: Pollock Icosahedral} and \ref{Theorem: Pollock Dodecahedral}.
\subsection*{Structure of the Paper} The paper is organized as follows. We use the Hardy--Littlewood circle method to prove Theorem \ref{Theorem: Representation Icosahedral}. Section \ref{sec: Initial Set-up} covers the initial setup required for this method. Sections \ref{sec: Minor Arcs} and \ref{sec: Minor Arcs II} contain details of the minor arc estimates. Sections \ref{sec: Major Arcs}, \ref{sec: Major Arcs II}, and \ref{sec: Major Arcs Sec III} are devoted to the major arc estimates. We present the proof of Theorem \ref{Theorem: Representation Icosahedral} in Section \ref{sec: Proof of Main Theorem}. Section \ref{sec: Proof of Theorem 1.6} is devoted to the proofs of Theorems \ref{Theorem: 8 Icosahedral Numbers} and \ref{Theorem: 8 Dodecahedral Numbers}, which in turn yield Theorems \ref{Theorem: Pollock Icosahedral} and \ref{Theorem: Pollock Dodecahedral} as corollaries. Since the methods pertaining to the icosahedral and dodecahedral cases are similar, to avoid redundancy, we provide brief outlines of the proofs of Theorems \ref{Theorem: 8 Dodecahedral Numbers} and \ref{Theorem: Pollock Dodecahedral}, highlighting the necessary adjustments. 

\subsection*{General Notation} We employ some standard notation that will be used throughout the article.

\begin{itemize}
    \item Throughout the paper, the expressions $f(X)=O(g(X))$, $f(X) \ll g(X)$, and $g(X) \gg f(X)$ are equivalent to the statement that $|f(X)| \leqslant C|g(X)|$ for all sufficiently large $X$, where $C>0$ is an absolute constant. 
    \item We define by $d(n)$, the divisor function defined on $\mathbb{N}$ by $d(n) = \sum_{d \mid n}1.$
    \item We write $e(\theta)$ to denote the expression $e^{2\pi i \theta}$.
    \item Given $\alpha \in \mathbb{R}$, the notation $\|\alpha\|$ denotes the smallest distance of $\alpha$ to an integer.
    \item Given a prime $p$, the notation $\mathbb{F}_{p}$ refers to the set of residue classes modulo $p$.
    \item We refer to $I_{n}$ as the $n$-th icosahedral number given by \eqref{Icosahedral Number : defn} and also employ the equivalent notations $I_t =\frac{5t^{3}}{2} - \frac{5t^{2}}{2} + t, t \in \mathbb{R} $ and $g(x) = \frac{5x^{3}}{2} - \frac{5x^{2}}{2} + x, x \in \mathbb{R}$ in Sections \ref{sec: Major Arcs II} and \ref{sec: Proof of Theorem 1.6} respectively to extend the definition over all real numbers.
\end{itemize}
\section{Setup for the Circle Method}\label{sec: Initial Set-up}
In 1918, Hardy and Ramanujan \cite{Hardy-Ramanujan} introduced the circle method to derive asymptotic estimates for the partition function. This method was subsequently developed further by Hardy and Littlewood, and later by Vinogradov. Here we will employ the Hardy--Littlewood circle method to prove Theorem \ref{Theorem: Representation Icosahedral}. For $m,s \in \mathbb{N}$, denote by $\mathcal{I}_s(m)$ the number of ways of writing a positive integer $m$ as the sum of $s$ icosahedral numbers. For $N \in \mathbb{N}$, $\alpha \in \mathbb{R}$, consider the sum
\begin{align}\label{Vinogradov's Sum}
f_N(\alpha,\mathcal{I}) = \sum_{n=1}^{N} e(\alpha I_n).
\end{align}
where $I_n$ is given by \eqref{Icosahedral Number : defn}. By abuse of notation, we write $f(\alpha) = f_N(\alpha,\mathcal{I})$, where $N$ is clearly implied. Then we have
\begin{align}
    f(\alpha)^s &= \sum_{n_1=1}^{N} \sum_{n_2=1}^{N} \cdots \sum_{n_s=1}^{N} e(\alpha (I_{n_1}+I_{n_2}+\cdots+I_{n_s})) \notag = \sum_{m=1}^{sI_N} \mathcal{I}_{s}(m, I_N) e(\alpha m),
\end{align}
where $\mathcal{I}_{s}(m, I_N)$ denotes the number of ways of writing a positive integer $m$ as the sum of $s$ icosahedral numbers, none of which is greater than $I_N$. Note that if $m \leqslant I_N$, then trivially $\mathcal{I}_s(m) = \mathcal{I}_{s}(m, I_N)$. Then Cauchy's integral formula
gives
\begin{align}\label{Orthogonality Relation}
  \int_0^1 f(\alpha)^s e(-\alpha m) \dd \alpha=\mathcal{I}_s(m).  
\end{align}
Let $m$ be large. We let
\begin{align}\label{Defining N}
N = \bigg \lceil \bigg (\frac{2m}{5}\bigg )^{\frac{1}{3}} \bigg \rceil +1, \quad P=N^\delta, 
\end{align}
where $\delta \in \mathbb{R}$ satisfies $N^{3\delta-3}<\frac{1}{2}$. When $1 \leqslant a \leqslant q \leqslant P$ and $(a, q)=1$, let
$$
\mathfrak{M}(q, a)=\left\{\alpha:|\alpha-a / q| \leqslant N^{\delta-3}\right\} .
$$
We call $\mathfrak{M}(q, a)$ the major arcs. Let $\mathfrak{M}$ denote the union of the $\mathfrak{M}(q, a)$'s. It is convenient to work on the unit interval $\mathfrak{U}=\left(N^{\delta-3}, 1+N^{\delta-3}\right]$ 
rather than $(0,1]$. This avoids any difficulties associated with having only `half major arcs' at 0 and 1 . Observe that $\mathfrak{M} \subset \mathfrak{U}$. The set $\mathfrak{m}=\mathfrak{U} \backslash \mathfrak{M}$ forms the minor arcs.
When $a / q \neq a^{\prime} / q^{\prime}$ and $1 \leqslant q, q^{\prime} \leqslant P$, one has
$$
\left|\frac{a}{q}-\frac{a^{\prime}}{q^{\prime}}\right| \geqslant \frac{1}{q q^{\prime}}>N^{-2\delta}>2 N^{\delta-3}.
$$
Thus, the $\mathfrak{M}(q, a)$ are pairwise disjoint. Using \eqref{Orthogonality Relation}, we have
\begin{align}\label{Major+Minor Arcs}
\mathcal{I}_s(m) = \int_{\mathfrak{M}} f(\alpha)^s e(-\alpha m) \dd \alpha+\int_{\mathfrak{m}} f(\alpha)^s e(-\alpha m) \dd \alpha.
\end{align}
To obtain an asymptotic expression for $\mathcal{I}_s(m)$, we estimate the major and minor arc integrals. We treat the integral over the minor arcs in Sections \ref{sec: Minor Arcs} and \ref{sec: Minor Arcs II}. The major arcs are addressed in Sections \ref{sec: Major Arcs}, \ref{sec: Major Arcs II}, and \ref{sec: Major Arcs Sec III}.

\section{Minor Arcs : Part I}\label{sec: Minor Arcs}
In this section, we establish some preliminary lemmas that will be needed in our treatment of the minor arcs.
\begin{lem}\label{Explicit Divisor Bound}
Let $d(n)$ denote the divisor function. Then for $n \geqslant e^3$, we have
\[ d(n) \leqslant n^{\frac{1.0661}{\log \log n}}.
\]
\end{lem}
\begin{proof}
See Nicolas--Robin \cite{nicolas1983majorations}.
\end{proof}
Let $\psi(x)$ be a real-valued function of $x$, and denote by $\Delta_1$, the forward difference operator
\begin{align}
\Delta_1(\psi(x) ; h)=\psi(x+h)-\psi(x).
\end{align}
We then define $\Delta_j$ for $j \geqslant 2$ recursively by means of the relation
$$
\begin{aligned}
\Delta_j(\psi(x) ; \mathbf{h}) =\Delta_j\left(\psi(x) ; h_1, \ldots, h_j\right) =\Delta_1\left(\Delta_{j-1}\left(\psi(x) ; h_1, \ldots, h_{j-1}\right) ; h_j\right) .
\end{aligned}
$$
By convention, we write $\Delta_0(\psi(x) ; h)=\psi(x)$. It's easy to see that when $1 \leqslant j \leqslant k$,
$$
\Delta_j\left(x^k ; \mathbf{h}\right)=h_1 \ldots h_j p_j\left(x ; h_1, \ldots, h_j\right),
$$
where $p_j$ is a polynomial in $x$ of degree $k-j$ with leading coefficient $k ! /(k-j)!$. By the linearity of the operator $\Delta_j$, it follows that
\[
\Delta_j\left(a_k x^k+\cdots+a_1 x ; \mathbf{h}\right)=\sum_{i=1}^k a_i \Delta_j\left(x^i ; \mathbf{h}\right).
\]
\begin{lem}\label{Weyl Differencing}
Let $\psi(x)$ be a real-valued arithmetic function, and suppose
$$
F(\psi)=\sum_{1 \leqslant x \leqslant X} e(\psi(x)).
$$
Then for each $j \in \mathbb{N}$,
$$
|F(\psi)|^{2^j} \leqslant(2 X)^{2^j-j-1} \sum_{\left|h_1\right|<X} \cdots \sum_{\left|h_j\right|<X} \sum_{x \in T_j(\mathbf{h})} e\left(\Delta_j(\psi(x) ; \mathbf{h})\right),
$$
where $T_j(\mathbf{h})$ denotes the interval of integers defined by putting $T_0(h)=[1, X]$, and for $j \geqslant 1$, we recursively set
$$
T_j\left(h_1, \ldots, h_j\right)=T_{j-1}\left(h_1, \ldots, h_{j-1}\right) \cap\left\{x \in[1, X]: x+h_j \in T_{j-1}\left(h_1, \ldots, h_{j-1}\right)\right\}.
$$
\end{lem}
\begin{proof}
See Vaughan \cite{vaughan1997hardy}.
\end{proof}

\begin{lem}\label{Geometric Sum}
Let $\alpha \in \mathbb{R}$. Let $X$ and $Y$ be real numbers with $Y>1$. Then
$$
\bigg \lvert \sum_{X<x \leqslant X+Y} e(\alpha x) \bigg \rvert \leqslant \min \left\{Y+1,\frac{1}{2}\|\alpha\|^{-1}\right\}.
$$
\end{lem}
\begin{proof}
The trivial upper bound for the sum is $Y+1$, which suffices to establish the desired conclusion when $\|\alpha\|=0$. When $\|\alpha\| \neq 0$, we deduce that
\begin{align}\label{GP Series}
\bigg \lvert \sum_{X<x \leqslant X+Y} e(\alpha x) \bigg \rvert & =\bigg \lvert \frac{e(\alpha\lfloor X+Y+1\rfloor)-e(\alpha\lfloor X+1\rfloor)}{e(\alpha)-1} \bigg \rvert \leqslant |\sin (\pi \alpha)|^{-1} .
\end{align}
The function $|\sin (\pi \alpha)|$ is periodic in $\alpha$ with period 1, and when $0 \leqslant \alpha \leqslant 1 / 2$, the inequality $2 \alpha \leqslant \sin (\pi \alpha) \leqslant \pi \alpha$ holds. Hence we have $|\sin (\pi \alpha)|^{-1}  \leqslant \frac{1}{2\|\alpha\|}$. Substituting this in \eqref{GP Series} and combining this with the trivial bound, the proof follows.
\end{proof}
Our next lemma is due to Vinogradov. We provide a complete proof here for the convenience of the reader.
\begin{lem}\label{Explicit Weyl's Inequality}
Let $X,Y \in \mathbb{R}$ with $X,Y \geqslant 1$. Let $\alpha,\beta \in \mathbb{R}$ and suppose there exist $a \in \mathbb{Z}$ and $q \in \mathbb{N}$ with $(a, q)=1$, $q>100$ and $|\alpha-a / q| \leqslant q^{-2}$. Then
$$
\sum_{1 \leqslant x \leqslant X} \min \left\{Y,\|\alpha x+\beta\|^{-1}\right\} \leqslant 8X Y\left(q^{-1}+Y^{-1}+X^{-1}+q(X Y)^{-1}\right) \log q .
$$
\end{lem}

\begin{proof}
Let $\theta=\alpha-a / q$, so that $|\theta| \leqslant 1 / q^2$. 
We break the range of summation into intervals of length $\lfloor q / 2\rfloor+1$, and consider a typical interval, say
$
J=\{n, n+1, \ldots, n+\lfloor q / 2\rfloor\} .
$
For any two distinct integers $n_1$ and $n_2$ with $n_2<n_1$ lying in $J$, we have
$$
\left\|\left(\alpha n_1+\beta\right)-\left(\alpha n_2+\beta\right)\right\|=\left\|\alpha\left(n_1-n_2\right)\right\| \geqslant\left\|\frac{a\left(n_1-n_2\right)}{q}\right\|-(n_1-n_2)|\theta| .
$$
Since $n_1 \neq n_2$ and $(n_1-n_2) \leqslant q / 2$, $q \nmid\left(n_1-n_2\right)$. Noting that $(a, q)=1$, we arrive at
$$
\left\|\left(\alpha n_1+\beta\right)-\left(\alpha n_2+\beta\right)\right\| \geqslant \frac{1}{q}-\frac{q / 2}{q^2}=\frac{1}{2 q} .
$$
Hence, mod 1 , $\alpha n_1+\beta$ and $\alpha n_2+\beta$ are spaced by a distance at least $(2 q)^{-1}$. Thus we have
$$
\begin{aligned}
\sum_{n \in J} \min \left\{Y,\|\alpha n+\beta\|^{-1}\right\} & \leqslant Y+\sum_{1<r \leqslant \lfloor q / 2\rfloor}\left|\frac{r}{2 q}\right|^{-1} \\
& \leqslant Y+2 \sum_{1 \leqslant r \leqslant q / 2} \frac{q}{r} \leqslant Y+4q \log q .
\end{aligned}
$$
Note that there are at most $\lceil X /(q / 2)\rceil$ intervals of the shape $J$ required to cover all of the summands $x$ with $1 \leqslant x \leqslant X$. Therefore we obtain
$$
\sum_{1 \leqslant x \leqslant X} \min \left\{Y,\|\alpha x+\beta\|^{-1}\right\} \leqslant (2X / q+1)(Y+4q \log q),
$$
and the conclusion of the lemma follows.
\end{proof}
The following corollary is an immediate consequence of Lemma \ref{Explicit Weyl's Inequality}.
\begin{coro}\label{Weyl's Inequality with Eta}
Let $X,Y \in \mathbb{R}$ with $X,Y \geqslant 1$. Let $\alpha,\beta \in \mathbb{R}$ and suppose there exist $a \in \mathbb{Z}$ and $q \in \mathbb{N}$ with $(a, q)=1$, $q>100$ and $|\alpha-a / q| \leqslant \eta q^{-2}$, for some absolute constant $\eta \geqslant 1$. Then
$$
\sum_{1 \leqslant x \leqslant X} \min \left\{Y,\|\alpha x+\beta\|^{-1}\right\} \leqslant 8X Y\eta \left(q^{-1}+Y^{-1}+X^{-1}+q(X Y)^{-1}\right) \log q .
$$
\end{coro}
\begin{proof}
The case when $\eta=1$ is treated in Lemma \ref{Explicit Weyl's Inequality}. For the general case, we work with intervals of length $\lfloor \frac{q}{2\eta}\rfloor+1$. The rest of the argument follows as in the proof of Lemma \ref{Explicit Weyl's Inequality}. 
\end{proof}
\begin{lem}
Let $\eta \geqslant 1$ be fixed and $X \geqslant e^3$. Let $\alpha_1, \alpha_2, \alpha_3 \in \mathbb{R}$. Suppose there exist $a \in \mathbb{Z}$ and $q \in \mathbb{N}$ which satisfy $(a, q)=1, q>100$ and $\left|\alpha_3-a / q\right| \leqslant \eta q^{-2}$. Then
$$
\bigg \lvert \sum_{1 \leqslant x \leqslant X} e\left(\alpha_1 x+\alpha_{2}x^2+\alpha_3 x^3\right) \bigg \rvert \leqslant 2X^{\frac{3}{4}}+8X^{1+\frac{0.53305}{\log (2\log X)}}\eta^{\frac{1}{4}}\left(q^{-1}+X^{-1}+q X^{-3}\right)^{\frac{1}{4}} (\log q)^{\frac{1}{4}}.
$$ \label{Weyl Inequality V2}
\end{lem} 
\begin{proof}
We write $\psi(x)=\alpha_1 x+ \alpha_{2}x^2 +\alpha_3 x^3$ and
$$
F(\boldsymbol{\alpha})=\sum_{1 \leqslant x \leqslant X} e(\psi(x)).
$$
Here by $\boldsymbol{\alpha}$, we denote the tuple $(\alpha_1,\alpha_2,\alpha_3)$. When $q>X^3$, the desired estimate is trivial. Hence, we may assume $q \leqslant X^3$. We apply Lemma \ref{Weyl Differencing} with $j=2$ to obtain the bound
$$
|F(\boldsymbol{\alpha})|^{4} \leqslant 2X \bigg \lvert \sum_{\left|h_1\right|<X} \sum_{\left|h_{2}\right|<X} \mathcal{E}(\mathbf{h}) \bigg \rvert,
$$
where 
$$
\mathcal{E}(\mathbf{h})=\sum_{x \in I_{2}(\mathbf{h})} e\left(\Delta_{2}(\psi(x) ; \mathbf{h})\right),
$$
and $I_{2}(\mathbf{h})$ is a suitable interval of integers contained in $[1, X]$. Note that $\Delta_{2}(\psi(x) ; \mathbf{h})=6 h_1 h_{2} x \alpha_3+r,$ where $r=r(\boldsymbol{\alpha} ; \mathbf{h})$ is independent of $x$. Thus, Lemma \ref{Geometric Sum} delivers the bound
$$
\mathcal{E}(\mathbf{h}) \leqslant \min \left\{X+1,\left\|6 h_1 h_{2} \alpha_3\right\|^{-1}\right\}.
$$
Hence, we arrive at
$$
|F(\boldsymbol{\alpha})|^{4} \leqslant 2X \bigg \lvert \sum_{\left|h_1\right|<X} \sum_{\left|h_{2}\right|<X} \min \left\{X+1,\left\|6 h_1 h_{2} \alpha_3\right\|^{-1}\right\} \bigg \rvert .
$$
Accounting for the summands in which $h_1 h_{2}=0$, we are led from here to the bound
\begin{align}\label{CS Inequality}
|F(\boldsymbol{\alpha})|^{4} \leqslant 2X\bigg(3X^{2}+4\sum_{1 \leqslant n \leqslant 6 X^2} d\left(\frac{n}{6}\right) \min \left\{X+1,\left\|n \alpha_3\right\|^{-1}\right\}\bigg) .
\end{align}
Here, the factor $d(n/6)$ only arises when $n$ is divisible by $6$. Invoking Lemma \ref{Explicit Divisor Bound} and Corollary \ref{Weyl's Inequality with Eta}, we therefore obtain the estimate
\begin{align}
|F(\boldsymbol{\alpha})|^{4} &\leqslant 6X^{3}+8X^{1+\frac{2.1322}{\log (2\log X)}} \sum_{1 \leqslant n \leqslant 6 X^2} \min \left\{X+1,\left\|n \alpha_3\right\|^{-1}\right\} \notag \\
& \leqslant 6X^{3}+1000 X^{4+\frac{2.1322}{\log (2\log X)}} \eta \left( q^{-1}+X^{-1}+qX^{-3}\right)\log q. \notag
\end{align}
This implies that
\begin{align}
F(\boldsymbol{\alpha}) &\leqslant 2X^{\frac{3}{4}}+8X^{1+\frac{0.53305}{\log (2\log X)}}\eta^{\frac{1}{4}}\left(q^{-1}+X^{-1}+q X^{-3}\right)^{\frac{1}{4}} (\log q)^{\frac{1}{4}}\notag,
\end{align}
which concludes the proof.
\end{proof}

\begin{rem}\label{Cauchy Schwarz Remark}
Although the bound provided in Lemma \ref{Weyl Inequality V2} is sufficient for handling the minor arcs, it is not optimal, especially when $q$ is very small relative to $X$, say for instance when
\[
1 \leqslant q \leqslant X^{\frac{0.53305}{\log (2\log X)}}.
\]
This is primarily because, even though the bound on the divisor function $d(n)$ in Lemma \ref{Explicit Divisor Bound} is optimal pointwise, it is quite weak in practice, as stronger bounds are known on average. To improve the bound in Lemma \ref{Weyl Inequality V2}, we first revisit the proofs of Lemma \ref{Explicit Weyl's Inequality} and Corollary \ref{Weyl's Inequality with Eta}. Following arguments similar to those proofs, we obtain
\begin{align}\label{Cauchy Schwarz Trick}
\sum_{1 \leqslant n \leqslant 6 X^2} \bigg (\min\left\{X+1,\left\|n \alpha_3\right\|^{-1}\right\}\bigg)^2 \ll \eta \bigg ( \frac{X^2}{q}+1 \bigg) (X^2+q^2).
\end{align}
Hence by the Cauchy--Schwarz inequality, the sum over $n$ in \eqref{CS Inequality} can be bounded as follows:
\begin{align}
\sum_{1 \leqslant n \leqslant 6 X^2}& d\left(\frac{n}{6}\right) \min \left\{X+1,\left\|n \alpha_3\right\|^{-1}\right\} \notag \\
&\ll \bigg (\sum_{1 \leqslant m \leqslant X^2} d(m)^2 \bigg)^{\frac{1}{2}} \bigg (\sum_{1 \leqslant n \leqslant 6 X^2} \bigg (\min\left\{X+1,\left\|n \alpha_3\right\|^{-1}\right\}\bigg)^2 \bigg)^{\frac{1}{2}}  \notag \\
&\ll X^3 (\log X)^{\frac{3}{2}} \eta^{\frac{1}{2}}(q^{-\frac{1}{2}}+X^{-1}+q X^{-2}), \label{CS Inequality Final}
\end{align}
where in the final step, we have used \eqref{Cauchy Schwarz Trick} and standard bounds for the second moment of the divisor function due to Ramachandra and Sankaranarayanan \cite{Ramachandra}. The bound in \eqref{CS Inequality Final} is non-trivial when $(\log X)^{3+\varepsilon} \leqslant q \leqslant X^{2-\varepsilon}$ for any fixed $\varepsilon>0$. Substituting \eqref{CS Inequality Final} into \eqref{CS Inequality}, we obtain the improved estimate
\[
F(\boldsymbol{\alpha})\ll X^{\frac{3}{4}}+X (\log X)^{\frac{3}{8}}\eta^{\frac{1}{8}}\left(q^{-\frac{1}{2}}+X^{-1}+q X^{-2}\right)^{\frac{1}{4}},
\]
which refines the bound in Lemma \ref{Weyl Inequality V2} when $q$ is small compared to $X$. However, for larger values of $q$, specifically when $q \approx X^{3-\delta}$, this estimate becomes inefficient, and we need to revert to the bound given in Lemma \ref{Weyl Inequality V2}. For instance, see \cite{bbz2025}, where similar techniques were employed to achieve additional savings in the analysis of the major arcs.
\end{rem}

\section{Minor Arcs : Part II}\label{sec: Minor Arcs II}
Our primary goal in this section is to prove the following theorem.
\begin{thm}\label{Minor Arcs : thm}
Let $f(\alpha)$ be as defined in \eqref{Vinogradov's Sum}. Then for $s \geqslant 9$ and $m \geqslant 10^{10}$, we have
$$
\int_{\mathfrak{m}} f(\alpha)^s e(-m \alpha) \dd \alpha \leqslant 152 \cdot 21^{s-8} \cdot (\log m)^{\frac{s-8}{4}}m^{\frac{s}{3}-1-\frac{\delta(s-8)}{12}+\frac{0.53305(s-8)+6.3966}{1.2\log \log m}}.
$$
\end{thm}
In order to establish Theorem \ref{Minor Arcs : thm}, we require the following lemma.
\begin{lem}\label{Inductive Lemma}
Let $f(\alpha)$ be as defined in \eqref{Vinogradov's Sum}. Suppose that $1 \leqslant j \leqslant 3$. Then for $N \geqslant e^3$, 
$$
\int_0^1|f(\alpha)|^{2^j} \dd \alpha \leqslant 152 N^{2^j-j+\frac{6.3966}{\log \log N}}.
$$
\end{lem}

\begin{proof} We consider the following cases.\\

\noindent \textbf{Case 1 : $j=1$.} It follows via orthogonality that
$$
\int_0^1|f(\alpha)|^2 \dd \alpha=\int_0^1 f(\alpha) f(-\alpha) \dd \alpha=\operatorname{card}\left\{1 \leqslant u,v \leqslant N: I_{u}=I_{v}\right\} \leqslant N .
$$\\
\noindent \textbf{Case 2 : $j=2$.} We apply Lemma \ref{Weyl Differencing} to see that
$$
|f(\alpha)|^{2} \leqslant\sum_{\left|h\right|<N} \sum_{n \in T_1(h)} e\left(\alpha \Delta_1\left(I_n ; h\right)\right),
$$
where $T_1(h)$ is a suitable subinterval of $[1, N]$. We write
\begin{align}\label{Inductive Step}
\int_0^1|f(\alpha)|^{4} \dd \alpha & =\int_0^1 f(\alpha) f(-\alpha)|f(\alpha)|^{2} \dd \alpha \leqslant S,
\end{align}
where
\begin{align}\label{S Definition}
S=\sum_{\left|h\right|<N} \sum_{n \in T_1(h)} \int_0^1 f(\alpha) f(-\alpha) e\left(\alpha \Delta_1\left(I_n ; h\right)\right) \dd \alpha .
\end{align}
By orthogonality, $S$ is bounded above by the number of integral solutions of the equation
$
\left(I_{u}-I_{v}\right)=\Delta_1\left(I_n ; h\right)
$
with $1 \leqslant u, v \leqslant N$, $1 \leqslant n \leqslant N$ and $\left|h\right|<N$.
The solutions counted by $S$ are of two types. First, there are the solutions in which $\left(I_{u}-I_{v}\right)=0$ which implies $\Delta_1\left(I_n ; h\right)=0$. By orthogonality, the number of choices of $u$ and $v$ here is
$$
\int_0^1 f(\alpha) f(-\alpha) \dd \alpha=\int_0^1|f(\alpha)|^{2} \dd \alpha \leqslant N.
$$
On the other hand, since $\Delta_1\left(I_n ; h\right)=0$, we have
$
h Q(n; h )=0,
$
where $Q(n;h)$ is a quadratic polynomial, determined by the choice of $h$. So either $h=0$, or else $n$ is a zero of $Q$. Then the total number of choices for $n$ and $h$ is $\leqslant5 N$. The contribution of the solutions of this first type to $S$ is therefore $
\leqslant 5N^2.
$
For the second type of solutions counted by $S$, we write
$
\left(I_{u}-I_{v}\right)=K,
$
for some non-zero integer $K=K(u,v)$ with $\lvert K \rvert \leqslant I_N.$ For each such choice of $u$ and $v$,
$
h Q(n; h ) =K,
$
and thus there are at most $2d(K)$ possible choices for $h$. Keeping $h$ fixed, there are at most $2$ choices for $n$. Therefore, the total possible choices for $h$ and $n$ are
\[ \leqslant 4 d(K) \leqslant 4 N^ {\frac{3.1983}{\log \log N}},
\]
when $N \geqslant e^3$. The contribution to $S$ from this second type of solution is therefore
$$
\leqslant 4\sum_{\substack{1 \leqslant u, v \leqslant N }} N^ {\frac{3.1983}{\log \log N}}  \leqslant 4 N^ {2+\frac{3.1983}{\log \log N}} .
$$
Combining both types, we have $S \leqslant 9 N^ {2+\frac{3.1983}{\log \log N}} $. Substituting this into \eqref{Inductive Step}, we conclude that
\begin{align}\label{Large Sieve Improvement}
  \int_0^1|f(\alpha)|^{4} \dd \alpha \leqslant 9 N^ {2+\frac{3.1983}{\log \log N}}.  
\end{align}
\\
\noindent \textbf{Case 3 : $j=3$.} Our approach is similar to the case when $j=2$. We have
\begin{align}\label{Inductive Step : j=3}
\int_0^1|f(\alpha)|^{8} \dd \alpha & =\int_0^1 f(\alpha)^{2} f(-\alpha)^{2}|f(\alpha)|^{4} \dd \alpha  \leqslant(2 N) S,
\end{align}
where
$$
S=\sum_{\left|h_1\right|<N} \sum_{\left|h_2\right|<N} \sum_{n \in T_2(\mathbf{h})} \int_0^1 f(\alpha)^{2} f(-\alpha)^{2} e\left(\alpha \Delta_2\left(I_n ; \mathbf{h}\right)\right) \dd \alpha,
$$
and $T_2(\mathbf{h})$ is a suitable subinterval of $[1, N]$. By orthogonality, the expression $S$ is bounded above by the number of integral solutions of the equation
$$
\sum_{i=1}^{2}\left(I_{u_i}-I_{v_i}\right)=\Delta_2\left(I_n ; \mathbf{h}\right)
$$
with $1 \leqslant u_i, v_i \leqslant N$, for all $1 \leqslant i \leqslant 2, 1 \leqslant n \leqslant N$ and $\left|h_j\right|<N$ for all $1 \leqslant j \leqslant 2$. The solutions counted by $S$ are of two types, depending on whether $\Delta_2\left(I_n ; \mathbf{h}\right)$ is equal to zero or not. When $\Delta_2\left(I_n ; \mathbf{h}\right)=0$, the contribution of the solutions to $S$ is consequently
$$
\leqslant8 N^2 \cdot 9 N^ {2+\frac{3.1983}{\log \log N}} \leqslant 72 N^ {4+\frac{3.1983}{\log \log N}}.
$$
When $\Delta_2\left(I_n ; \mathbf{h}\right) \neq 0$, following a similar approach to the case when $j=2$, the contribution of the solutions to $S$ is
$$
\leqslant 4\sum_{\substack{1 \leqslant u_i, v_i \leqslant N \\ 1 \leqslant i \leqslant 2}}  N^ {\frac{6.3966}{\log \log N}} \leqslant 4 N^{4+\frac{6.3966}{\log \log N}}.
$$
Combining the two estimates, $S \leqslant 76 N^{4+\frac{6.3966}{\log \log N}}$. Putting this into \eqref{Inductive Step : j=3}, the proof follows.
\end{proof}
\begin{rem}\label{Large Sieve Remark}
Similar to Remark \ref{Cauchy Schwarz Remark}, it is possible to improve the bounds in Lemma \ref{Inductive Lemma} by using mean-value estimates of the divisor function rather than pointwise bounds. Indeed, for the case $j=2$ in Lemma \ref{Inductive Lemma}, following \eqref{S Definition}, we have
\begin{align*}
    S \leqslant \sum_{\left|h\right|<N} \sum_{1 \leqslant n \leqslant N} \sum_{\substack{1 \leqslant u,v \leqslant N \\ I_u-I_v = h Q(n,h)}} 1,
\end{align*}
where $Q(n,h)$ is a quadratic polynomial determined by the choice of $h$. For $h,u$ and $v$ fixed, there are at most two possible choices of $n$. Therefore, we obtain
\begin{align}\label{Large Sieve Step 1}
S \ll \sum_{\left|h\right|<N} \sum_{\substack{1 \leqslant u,v \leqslant N \\ I_u-I_v \equiv 0 \bmod h}} 1.
\end{align}
Now, we use the orthogonality of additive characters to derive
\begin{align}\label{Large Sieve Step 2}
S \ll \sum_{\left|h\right|<N} \sum_{\ell=0}^{h-1}  \sum_{\substack{1 \leqslant u,v \leqslant N}} e \left ( \frac{\ell(I_u-I_v)}{h} \right) = \sum_{\left|h\right|<N} \sum_{\ell=0}^{h-1} \bigg \lvert \sum_{1 \leqslant u \leqslant N} \frac{1}{\sqrt{h}} e \left ( \frac{\ell I_u}{h}\right)\bigg \rvert^2.
\end{align}
The expression on the right-hand side of \eqref{Large Sieve Step 2} is now suited for the application of the large sieve inequality for integer polynomial amplitudes; see Prakash–Ramana \cite[Theorem 1]{LargeSieveInequality}. Indeed, following the notation in \cite[Theorem 1]{LargeSieveInequality}, we set $P(i) = 5i^3-5i^2+2i$. Then the right-hand side of \eqref{Large Sieve Step 2} is
\begin{align}\label{Large Sieve Step 3}
&\ll \sum_{d \leqslant N} \sum_{h' \leqslant 2N/d} \sum_{\substack{\ell'=1 \\ (\ell',h')=1}}^{h'} \bigg \lvert \sum_{1 \leqslant i \leqslant N} \frac{1}{\sqrt{h'd}} e \left ( \frac{\ell' P(i)}{h'}\right)\bigg \rvert^2 \notag \\
&\ll \sum_{d \leqslant N} \sum_{x \in \mathcal{F}(2N/d)} \bigg \lvert \sum_{1 \leqslant i \leqslant N} \frac{1}{\sqrt{h'd}} e \left ( xP(i)\right)\bigg \rvert^2 \ll N^2(\log N)^{20}.
\end{align}
Here $\mathcal{F}(2N/d)$ is the Farey sequence of order $2N/d$, that is, all reduced fractions $a/q$ in the interval $[0,1]$ with $1 \leqslant q \leqslant 2N/d$. Substituting \eqref{Large Sieve Step 3} in \eqref{Inductive Step}, we see that
\[
\int_0^1|f(\alpha)|^{4} \dd \alpha \ll N^2(\log N)^{20}.
\]
which yields an improvement over the bound in \eqref{Large Sieve Improvement}. A concomitant argument holds for the case $j=3$. Since in Theorem \ref{Theorem: Representation Icosahedral} we desire to obtain concrete, explicit error terms, we will ultimately make use of the bound given in Lemma \ref{Large Sieve Improvement}. 
\end{rem}
\begin{proof}[Proof of Theorem \ref{Minor Arcs : thm}]
We can write
$$
\begin{aligned}
\left|\int_{\mathfrak{m}} f(\alpha)^s e(-m \alpha) \dd \alpha\right| & \leqslant\left(\sup _{\alpha \in \mathfrak{m}}|f(\alpha)|\right)^{s-8} \int_0^1|f(\alpha)|^{8}  \dd  \alpha.
\end{aligned}
$$
Consider an arbitrary point $\alpha$ of $\mathfrak{m}$. By Dirichlet's Theorem (see \cite[Lemma 2.1]{vaughan1997hardy}), there exist $a, q$ with $(a, q)=1$ and $q \leqslant N^{3-\delta}$ such that $|\alpha-a / q| \leqslant q^{-1} N^{\delta-3}$. Since $\alpha \in \mathfrak{m} \subset\left(N^{\delta-3}, 1-N^{\delta-3}\right)$ it follows that $1 \leqslant a \leqslant q$. Therefore $q>N^\delta$, for otherwise $\alpha$ would be in $\mathfrak{M}$. We now apply Lemma \ref{Weyl Inequality V2}. Choose $\eta = 10$. Then
\begin{align}
\lvert f(\alpha) \rvert &\leqslant  2N^{\frac{3}{4}}+8 \cdot (30)^{\frac{1}{4}} N^{1+\frac{0.53305}{\log \log N}}\left(q^{-1}+N^{-1}+q N^{-3}\right)^{\frac{1}{4}} (\log N)^{\frac{1}{4}} \notag \\
&\leqslant  2N^{\frac{3}{4}}+8 \cdot (90)^{\frac{1}{4}} N^{1-\frac{\delta}{4}+\frac{0.53305}{\log \log N}} (\log N)^{\frac{1}{4}} \notag \\
& \leqslant 27  N^{1-\frac{\delta}{4}+\frac{0.53305}{\log \log N}} (\log N)^{\frac{1}{4}}. \label{f_alpha bound}
\end{align}
From Lemma \ref{Inductive Lemma}, we have
\begin{align}
\int_0^1|f(\alpha)|^{8} \dd \alpha \leqslant 152 N^{5+\frac{6.3966}{\log \log N}}. \label{f_alpha Integral Bound}
\end{align}
Combining \eqref{f_alpha bound} and \eqref{f_alpha Integral Bound}, we arrive at
$$
\begin{aligned}
\left|\int_{\mathfrak{m}} f(\alpha)^s e(-m \alpha) \dd \alpha\right| &  \leqslant 152 \left(27 N^{1-\frac{\delta}{4}+\frac{0.53305}{\log \log N}}\right)^{s-8} (\log N)^{\frac{s-8}{4}} N^{5+\frac{6.3966}{\log \log N}} \\
& \leqslant 152 \cdot 27^{s-8} \cdot (\log N)^{\frac{s-8}{4}}N^{s-3-\frac{\delta(s-8)}{4}+\frac{0.53305(s-8)+6.3966}{\log \log N}} \\
& \leqslant 152 \cdot 21^{s-8} \cdot (\log m)^{\frac{s-8}{4}}m^{\frac{s}{3}-1-\frac{\delta(s-8)}{12}+\frac{0.53305(s-8)+6.3966}{1.2\log \log m}},
\end{aligned}
$$
which completes the proof.
\end{proof}
\section{Major Arcs : Initial Steps}\label{sec: Major Arcs}
In this section, our main goal is to effectively approximate $f(\alpha)$ when $\alpha \in \mathfrak{M}$. Let $\theta = \alpha-a/q$.
Since $\mathfrak{M}(q, a)$ are pairwise disjoint, we can write
\begin{align}\label{Major Arc Disjointness}
\int_{\mathfrak{M}} f(\alpha)^s e(-\alpha m)d\alpha= \sum_{q\leqslant N^\delta}\sum\limits_{\substack{a=1 \\ (a,q)=1}}^q\int_{\mathfrak{M}(q,a)} f(\alpha)^se(-\alpha m) \dd\alpha.
\end{align} 
Define 
\begin{align}
A(t) &:=  \sum_{1\leqslant n\leqslant t} e\bigg (\frac{a}{q}I_n \bigg), \label{Definition A(t)} \\
\textrm{and} \quad V(q,a) &:= \sum_{1\leqslant n\leqslant 2q} e\bigg (\frac{a}{q}I_n \bigg) \label{Definition V(q,a)}.
\end{align}
By abuse of notation, we extend \eqref{Icosahedral Number : defn} by writing $I_t$ for $t \in \mathbb{R}$. More precisely, we write
\[
I_t = \frac{5}{2}t^3-\frac{5}{2}t^2+t \quad \textrm{and} \quad I_t' = \frac{15}{2}t^2-5t+1, \quad t \in \mathbb{R}.
\]
\begin{lem}\label{Partial Summation}
Let $\alpha \in \mathfrak{M}(q,a)$ and $\theta = \alpha - a/q$. Then 
\begin{align}
 f(\alpha) &=A(N)e(\theta I_N)-2\pi i\theta\int_{1}^N A(t) I_t'e(\theta I_t) \dd t.
\end{align}
\end{lem}
\begin{proof}
The proof follows by applying partial summation.
\end{proof}
\begin{lem}\label{Congruence}
If $q$ is even and $n$ is odd, then $I_n\equiv I_{n+2q}$ mod $q$. Otherwise, $I_n\equiv I_{n+q}$ mod $q$.
\end{lem}
\begin{proof} The congruence relation $I_n\equiv I_{n+2q} \bmod q$ holds true always. On the other hand, $I_n\equiv I_{n+q} \bmod q$ if and only if $15n(n+q)$ is even. The desired conclusion is now immediate.
\end{proof}
\begin{lem}\label{Approx 1}
For all $1 \leqslant t \leqslant N$, with $A(t)$ defined as in \eqref{Definition A(t)}, we have
\begin{align*}
   \bigg \lvert A(t) - \frac{V(q,a)}{2q} t \bigg \rvert \leqslant 2q.
\end{align*}
\end{lem}
\begin{proof}
By Lemma \ref{Congruence}, we have
\begin{align}
    A(t) & = \sum_{n=1}^{2q} e\bigg (\frac{a}{q}I_n \bigg)\floor*{\frac{t}{2q}} +\sum_{n = \floor*{\frac{t}{2q}}2q+1}^t e \bigg (\frac{a}{q}I_n \bigg) \notag\\
    &= \frac{V(q,a)}{2q} t-\sum_{n=1}^{2q}e\bigg (\frac{a}{q}I_n \bigg)\bigg \{\frac{t}{2q}\bigg \}+ \sum_{n=1}^{t-\floor*{\frac{t}{2q}}2q}e\bigg (\frac{a}{q}I_n \bigg) \notag \\
    &=\frac{V(q,a)}{2q} t + \sum_{n=1}^{t-\floor*{\frac{t}{2q}}2q}e\bigg (\frac{a}{q}I_n \bigg)\bigg(1-\bigg \{\frac{t}{2q}\bigg \}\bigg )-\sum_{n=t-\floor*{\frac{t}{2q}}2q+1}^{2q}e\bigg (\frac{a}{q}I_n \bigg)\bigg \{\frac{t}{2q}\bigg\}. \label{Approx 1 Prefinal}
\end{align}
Upon trivially bounding the two sums on the far right side of \eqref{Approx 1 Prefinal}, the proof follows.
\end{proof} 
\begin{lem}\label{Approx 2}
Let $\alpha \in \mathfrak{M}(q,a)$ and $\theta = \alpha - a/q$.
Then
\[
\bigg \lvert f(\alpha)-\frac{V(q,a)}{2q}\int_1^N e(\theta I_t)\dd t \bigg \rvert \leqslant 2q+1+10q\pi\theta N^3.
\]
\end{lem}
\begin{proof}
Applying Lemma \ref{Approx 1}, we have the following two bounds:
\begin{align*}
    \bigg \lvert A(N)e(\theta I_N)-\frac{V(q,a)}{2q} N e(\theta I_N)\bigg \rvert 
   &\leqslant 2q,\\
\bigg \lvert 2\pi i\theta \int_{1}^NA(t)I_t'e(\theta I_t)\dd t-2\pi i\theta\int_{1}^N\frac{V(q,a)}{2q}tI_t'e(\theta I_t)\dd t \bigg \rvert &\leqslant 4q\pi\theta\int_{1}^N I_t'\dd t \leqslant 10q\pi\theta N^3.
\end{align*}
Therefore, by Lemma \ref{Partial Summation} and the triangle inequality,
\begin{align} \label{Approx 2 Step 1}
\bigg \lvert f(\alpha)-\frac{V(q,a)}{2q}Ne(\theta I_N)+2\pi i\theta\int_{1}^N&\frac{V(q,a)}{2q}t I_t'e(\theta I_t)\dd t \bigg \rvert \leqslant 2q+10q\pi\theta N^3.
\end{align}
On the other hand, applying integration by parts,
\begin{align}
\frac{V(q,a)}{2q} &N e(\theta I_N)-2\pi i\theta\int_{1}^N\frac{V(q,a)}{2q}tI_t'e(\theta I_t) \dd t \notag \\
&=\frac{V(q,a)}{2q}e(\theta)+\frac{V(q,a)}{2q}\int_1^N e(\theta I_t) \dd t. \label{Approx 2 Step 2}
\end{align}
Combining \eqref{Approx 2 Step 1} and \eqref{Approx 2 Step 2} and trivially bounding $V(q,a)$, we obtain the desired result.
\end{proof}
\begin{lem}\label{Approx 3}
Let $N \geqslant 100, \alpha \in \mathfrak{M}(q,a)$ and $\theta = \alpha - a/q$. Then
\[
\bigg \lvert f(\alpha)-\frac{V(q,a)}{2q}\int_1^N e\bigg( \frac{5t^3 \theta}{2} \bigg) \dd t \bigg \rvert \leqslant 2q+1+10q\pi\theta N^3+12\pi N^\delta.
\]    
\end{lem}
\begin{proof}
By Lemma \ref{Approx 2}, it suffices to show that
\[
\bigg \lvert \int_1^Ne(\theta I_t) \dd t-\int_1^N e\bigg( \frac{5t^3 \theta}{2} \bigg) \dd t \bigg \rvert \leqslant 12\pi N^\delta.
\]
We write
\begin{align}\label{Approx 3 Step 1}
\bigg \lvert \int_1^Ne(\theta I_t) \dd t-\int_1^N e\bigg( \frac{5t^3 \theta}{2} \bigg) \dd t \bigg \rvert \leqslant \int_1^N \bigg \lvert e\bigg(-\frac{5t^2 \theta}{2}+t\theta\bigg)-1\bigg \rvert \dd t.
\end{align}
Since $|\theta|\leqslant N^{\delta-3}$ and $1 \leqslant t \leqslant N,$ we have for any $N \geqslant 100$,
\[
\bigg \lvert 2\pi\theta\bigg(-\frac{5t^2}{2}+t\bigg) \bigg \rvert \leqslant 6\pi N^{\delta-1} \leqslant 1.
\]
Therefore, by Taylor expansion, we obtain
\begin{align}
\bigg \lvert e\bigg(-\frac{5t^2 \theta}{2}+t\theta\bigg)-1\bigg \rvert &\leqslant \sum_{n=1}^\infty \frac{|2\pi\theta(-\frac{5}{2}t^2+t)|}{n!} \leqslant \bigg \lvert 4\pi\theta\bigg(-\frac{5t^2}{2}+t\bigg) \bigg \rvert \leqslant 12\pi N^{\delta-1}. \label{Approx 3 Step 2}
\end{align}
Substituting \eqref{Approx 3 Step 2} into \eqref{Approx 3 Step 1} and integrating, we obtain the desired result.
\end{proof}
\begin{lem}\label{Approx 4}
Let $N \geqslant 100$. Then
\[
\bigg \lvert \int_{\mathfrak{M}}f(\alpha)^s e(-\alpha m) \dd\alpha -\int_{\mathfrak{M}} \bigg (\frac{V(q,a)}{2q}\int_1^N e\bigg( \frac{5t^3 \theta}{2} \bigg) \dd t \bigg)^s  e(-\alpha m) \dd\alpha \bigg \rvert \leqslant 113s N^{5 \delta+s-4}.
\]
\end{lem}
\begin{proof}
An application of the Binomial Theorem combined with Lemma \ref{Approx 3} shows that 
\begin{align*}
\bigg \lvert f(\alpha)^s-\bigg (\frac{V(q,a)}{2q}\int_1^N e\bigg( \frac{5t^3 \theta}{2} \bigg) \dd t \bigg)^s \bigg \rvert \leqslant s N^{s-1}\left(2q+1+10q\pi\theta N^3+12\pi N^{\delta}\right).
\end{align*}
Since the $\mathfrak{M}(q,a)$'s are disjoint, we therefore obtain
\begin{align}
\bigg \lvert \int_{\mathfrak{M}}f(\alpha)^s &e(-\alpha m) \dd\alpha -\int_{\mathfrak{M}} \bigg (\frac{V(q,a)}{2q}\int_1^N e\bigg( \frac{5t^3 \theta}{2} \bigg) \dd t \bigg)^s  e(-\alpha m) \dd\alpha \bigg \rvert \notag \\
& \leqslant \sum_{q\leqslant N^\delta}\sum\limits_{\substack{a=1 \\ (a,q)=1}}^q\int_{-N^{\delta-3}}^{N^{\delta-3}}s\left((2q+1)N^{s-1}+10q\pi\theta N^{s+2}+12\pi N^{\delta+s-1}\right) \dd\theta \notag\\
&\leqslant 2s\sum_{q\leqslant N^\delta}\sum\limits_{\substack{a=1 \\ (a,q)=1}}^q\left(3N^{s+2\delta-4}+5\pi N^{s+3\delta-4}+12\pi N^{2\delta+s-4}\right) \leqslant 113sN^{5\delta+s-4},\notag
\end{align}
which completes the proof.
\end{proof} 
Suppose $s \geqslant 9$. Let $\alpha \in \mathfrak{M}(q,a)$, $\theta = \alpha - a/q$ and define 
\begin{align}
\mathfrak{S}(m, Q) &: =\sum_{q\leqslant Q}\sum\limits_{\substack{a=1 \\ (a,q)=1}}^q\bigg(\frac{V(q,a)}{2q}\bigg)^s e\bigg(-\frac{am}{q}\bigg), \label{S (m,Q) definition} \\
v(\theta) &: = \int_1^N e\bigg (\frac{5t^3 \theta}{2} \bigg) \dd t, \label{v theta definition} \\
\textrm{and} \quad J^*(m) &:= \int_{-N^{\delta-3}}^{N^{\delta-3}} \left(\left(\frac{5}{2}\right)^{1/3}v(\theta)\right)^se(-\theta m) \dd\theta. \label{J* definition}
\end{align}
Furthermore, we let 
\begin{align}
\mathcal{I}_{s}^*(m)&:=\int_{\mathfrak{M}} \bigg (\frac{V(q,a)}{2q}\int_1^N e\bigg( \frac{5t^3 \theta}{2} \bigg) \dd t \bigg)^s  e(-\alpha m) \dd\alpha \notag\\
&=\sum_{q\leqslant N^\delta}\sum\limits_{\substack{a=1 \\ (a,q)=1}}^q\bigg(\frac{V(q,a)}{2q}\bigg)^s e\bigg(-\frac{am}{q}\bigg)\int_{-N^{\delta-3}}^{N^{\delta-3}} v(\theta)^se(-\theta m) \dd\theta \notag\\
&=\left(\frac{2}{5}\right)^{s/3}\mathfrak{S}(m,N^\delta)J^*(m) \label{Approximating Major Arc Integral}.
\end{align}
We will approximate our major arc integral by $\mathcal{I}_{s}^*(m)$. To do so, we will estimate $\mathfrak{S}(m,N^\delta)$ and $J^*(m)$ separately. We will accomplish this in Sections \ref{sec: Major Arcs II} and \ref{sec: Major Arcs Sec III} respectively.

\section{Major Arcs : The Singular Series}\label{sec: Major Arcs II} 
\subsection{Completing the Singular Series} We first complete the series $\mathfrak{S}(m,N^\delta)$. Define
\begin{align}
    \mathfrak{S}(m) &:= \sum_{q=1}^\infty V(q)\label{definition of S(m)}, \quad \textrm{where} \\
    \label{V Definition}
    V(q) &:= \sum\limits_{\substack{a=1 \\ (a,q)=1}}^q\bigg(\frac{V(q,a)}{2q}\bigg)^se\bigg(-\frac{am}{q}\bigg),
\end{align}
and $V(q,a)$ is given by \eqref{Definition V(q,a)}. We first show that $V(q)$ is multiplicative.
\begin{lem}\label{V Multiplicativity 1}
Suppose $(a,q)=(b,r)=(q,r)=1$. Then $V(qr, ar+bq)=\frac{1}{2}V(q,a)V(r,b)$.
\end{lem}
\begin{proof} We only consider the case when both $q$ and $r$ are odd. A concomitant argument holds for the other cases. By Lemma \ref{Congruence},
\begin{align*}
V(qr,ar+bq)=2\sum_{n=1}^{qr} e\bigg(\frac{ar+bq}{qr}I_n\bigg).
\end{align*}
By Euclid's algorithm, each residue class $m$ modulo $qr$ can be represented uniquely in the form $tr+uq$ with $1\leqslant t\leqslant q$ and $1\leqslant u\leqslant r$. Therefore, we obtain 
\begin{align*}
     V(qr,ar+bq) &=2\sum_{t=1}^{q}\sum_{u=1}^r e\bigg(\frac{ar+bq}{qr}I_{tr+uq}\bigg)\\
     &=2\sum_{t=1}^{q}\sum_{u=1}^re\bigg(\frac{ar+bq}{qr}\cdot(I_{tr}+I_{uq})\bigg)\\
     &=2\sum_{t=1}^{q}e\bigg(\frac{a}{q}I_{tr}\bigg)\sum_{u=1}^re\bigg(\frac{b}{r}I_{uq}\bigg).
\end{align*}
Since $tr$ and $uq$ runs over complete residue classes modulo $q$ and $r$ respectively, we deduce that
\begin{align*}
V(qr,ar+bq)= 2\sum_{t=1}^qe\bigg(\frac{a}{q}I_t\bigg)\sum_{u=1}^re\bigg(\frac{b}{r}I_u\bigg)=\frac{1}{2}V(q,a)V(r,b),
\end{align*}
where the last equality is obtained by another application of Lemma \ref{Congruence}.
\end{proof}
\begin{lem}\label{V Multiplicativity 2}
The function $V(q)$, as defined in \eqref{V Definition}, is multiplicative.
\end{lem}
\begin{proof}
Note that $V(1)=1$. Suppose $(q,r)=1$. Then by Lemma \ref{V Multiplicativity 1},
    \begin{align*}
         V(qr) &= \sum\limits_{\substack{a=1 \\ (a,qr)=1}}^{qr}\bigg(\frac{V(qr,a)}{2qr}\bigg)^se\bigg(-\frac{am}{qr}\bigg)\\
         &=\sum\limits_{\substack{a=1 \\ (a,q)=1}}^{q}\sum\limits_{\substack{b=1 \\ (b,r)=1}}^{r}\bigg(\frac{V(qr,ar+bq)}{2qr}\bigg)^se\bigg(-\frac{ar+bq}{qr}m\bigg)\\
         &=\sum\limits_{\substack{a=1 \\ (a,q)=1}}^{q}\sum\limits_{\substack{b=1 \\ (b,r)=1}}^{r}\bigg(\frac{V(q,a)V(r,b)}{4qr}\bigg)^se\bigg(-\frac{am}{q}\bigg)e\bigg(-\frac{bm}{r}\bigg)\\
          &=V(q)V(r),
    \end{align*}
    which concludes the proof.
\end{proof}
\begin{lem}\label{Singular Series Extension}
Let $s \geqslant 9$ and $N^{\delta} \geqslant e^{e^{45}}$. Then $|\mathfrak{S}(m)|\leqslant e^{e^{46}}$. Furthermore, we have
\begin{align*}
|\mathfrak{S}(m)-\mathfrak{S}(m, N^\delta)| \leqslant \frac{ 24^s}{\left(\frac{5s}{21}-2\right) N^{\left(\frac{5s}{21}-2\right)\delta}}.
\end{align*}
\end{lem}
\begin{proof}
We first evaluate $V(q,a)$ by applying Lemma \ref{Weyl Inequality V2}. To cover both cases $(5a,2q)=1$ and $(5a,2q)\neq 1$, it suffices to choose $\eta = 25$ in Lemma \ref{Weyl Inequality V2}. We obtain
\begin{align*}
|V(q,a)|&\leqslant 2(2q)^{\frac{3}{4}}+8 \cdot (25)^{\frac{1}{4}}(2q)^{1+\frac{0.53305}{ \log (2\log 2q)}}\bigg (\frac{1}{q}+\frac{1}{2q}+\frac{1}{8q^2}\bigg )^{\frac{1}{4}}(\text{log }q)^{\frac{1}{4}}\\
 &\leqslant 3.5 q^{\frac{3}{4}}+16 \cdot (50)^{\frac{1}{4}} \cdot 2^{ \frac{0.53305}{\log (2\log 2q)}} q^{\frac{3}{4}+\frac{0.53305}{\log (2\log 2q)}+ \frac{\log \log q}{4\log q}} \\
&\leqslant 3.5 q^{\frac{3}{4}}+43  q^{\frac{3}{4}+\frac{0.53305}{\log (2\log 2q)}+ \frac{\log \log q}{4\log q}} .
\end{align*}
When $q \geqslant e^{e^{45}}$, we have
\[
\frac{0.53305}{\log (2\log 2q)}+ \frac{\log \log q}{4\log q} \leqslant \frac{1}{84},
\]
which implies that
\[
|V(q,a)| \leqslant 3.5 q^{\frac{3}{4}}+43 q^{\frac{16}{21}} \leqslant 47 q^{\frac{16}{21}}.
\]
Thus, it follows that $ |V(q)| \leqslant 24^s q^{1-\frac{5s}{21}}$ for $q \geqslant e^{e^{45}}$. Since $s \geqslant 9$, this ensures that $\mathfrak{S}(m)$ converges absolutely and uniformly with respect to $m$. We have 
\begin{align*}
|\mathfrak{S}(m)| \leqslant\sum_{q=1}^\infty|V(q)| &\leqslant\sum_{q=1}^{\lfloor e^{e^{45}} \rfloor}|V(q)|+ 24^s\int_{e^{e^{45}}}^\infty x^{1-\frac{5s}{21}}\dd x\\
    & \leqslant \sum_{q=1}^{\lfloor e^{e^{45}} \rfloor} q+\frac{24^s}{\frac{5s}{21}-2}\left(e^{e^{45}}\right)^{2-\frac{5s}{21}} \leqslant e^{e^{46}}.
\end{align*}
Moreover, when $N^{\delta} \geqslant e^{e^{45}}$,
we deduce that
\begin{align*}
    |\mathfrak{S}(m)-\mathfrak{S}(m, N^\delta)| \leqslant 24^s\int_{N^{\delta}}^\infty x^{1-\frac{5s}{21}} \dd x \leqslant  \frac{ 24^s}{\left(\frac{5s}{21}-2\right) N^{\left(\frac{5s}{21}-2\right)\delta}},
\end{align*}
which is the desired conclusion.
\end{proof}
\subsection{Counting Solutions to Arithmetic Congruences}\label{subsec: Arithmetic Congruences} Since $\mathfrak{S}(m)$ converges absolutely by Lemma \ref{Singular Series Extension} and $V(q)$ is multiplicative, we have the Euler product
\begin{align*}
    \mathfrak{S}(m) =\prod_{p\text{ prime}}\sum_{k=0}^\infty V(p^k)=\prod_{p\text{ prime}} (1+V(p)+V(p^2)+\cdots).
\end{align*}
Suppose $1\leqslant n_i\leqslant t$. Let $\mathcal{M}_m(t,q)$ be the number of solutions of the congruence equation
\begin{align}\label{f(x) formulas}
g(n_1)+g(n_2)+\cdots+g(n_s)\equiv m \bmod q,
\end{align}
where $g(x)=\frac{5}{2}x^3-\frac{5}{2}x^2+x$. By abuse of notation, we write $\mathcal{M}_m(q,q) = \mathcal{M}_m(q)$.\\

In order to estimate the number of solutions, one option would be to make use of the Lang--Weil Theorem \cite{langweil} on counting points on varieties over finite fields. However, the square root of $p$ saving provided by their theorem is not enough for our purposes. So we will proceed differently in order to obtain a power saving larger than a full factor of $p$. We begin with the following lemma.

\begin{lem}\label{V Multiplicativity 3}
For $q \in \mathbb{N}$, we have
\[
\sum_{d\mid q}V(d)=q^{1-s}2^{-s}\mathcal{M}_m(2q,q). 
\]
\end{lem}
\begin{proof} Using the orthogonality relation
\[
\mathbbm{1}_{n \equiv m \bmod q} = \frac{1}{q} \sum_{\ell=1}^{q} e \left (\frac{ \ell(n-m)}{q} \right),
\]
we can write
\begin{align*}
\mathcal{M}_m(2q,q)=\frac{1}{q}\sum_{r=1}^q\sum_{n_1=1}^{2q}\sum_{n_2=1}^{2q}\cdots\sum_{n_s=1}^{2q}e(r(g(n_1)+g(n_2)+\cdots+g(n_s)-m)/q).
\end{align*}
The sum over $r$ can be rearranged into subsums according to $(r,q)$. Then the general term in each subsum is a periodic function with period $q/(r,q)=d$, say. Hence we obtain
\begin{align*}
\mathcal{M}_m(2q,q)&=\frac{1}{q}\sum_{d\mid q}\sum\limits_{\substack{a=1 \\ (a,d)=1}}^d\bigg (\frac{q}{d}\bigg)^s\sum_{n_1=1}^{2d}\cdots\sum_{n_s=1}^{2d}e(a(g(n_1)+g(n_2)+\cdots+g(n_s)-m)/d)\\
  &=q^{s-1}2^s\sum_{d\mid q}\sum\limits_{\substack{a=1 \\ (a,d)=1}}^d\bigg(\frac{1}{2d}\bigg)^s\sum_{n_1=1}^{2d}\cdots\sum_{n_s=1}^{2d}e(a(g(n_1)+g(n_2)+\cdots+g(n_s)-m)/d)\\ 
  &=q^{s-1}2^s\sum_{d\mid q}V(d).
\end{align*}
Thus, we have $\sum_{d\mid q}V(d)=q^{1-s}2^{-s}\mathcal{M}_m(2q,q)$.
\end{proof}
Observe that by choosing $q=p^k$, Lemma \ref{V Multiplicativity 3} yields
\begin{align*}
    \mathfrak{S}(m) &=\prod_{p\textnormal{ prime}}\sum_{k=0}^\infty V(p^k)=\prod_{p\textnormal{ prime}} \lim_{k\rightarrow\infty} 2^{-s}p^{k(1-s)}\mathcal{M}_m(2p^k,p^k).
\end{align*}
Let 
\begin{align}\label{T Definition}
T_m(p):= \lim_{k\rightarrow\infty} p^{k(1-s)}\mathcal{M}_m(p^k).
\end{align}
Since $\mathcal{M}_m(2q,q) = 2^s \mathcal{M}_m(q)$, we obtain
\begin{align}\label{Crucial Limit Result}
\mathfrak{S}(m) = \prod_{p\text{ prime}} \lim_{k\rightarrow\infty} p^{k(1-s)}\mathcal{M}_m(p^k)= \prod_{p\text{ prime}} T_m(p).
\end{align}
Therefore, in order to estimate $\mathfrak{S}(m)$, we will need to evaluate $T_m(p)$.
\begin{lem}\label{T_m(p) complete bound}
For any $s \geqslant 9$ and any prime $p$, we have
\begin{align*}
\left|T_m(p)-1\right| \leqslant e^{se^{89}}\left(\frac{p - 1}{p}\right)\cdot \left( \frac{p^{1 - \frac{5s}{21}}}{1 - p^{1 - \frac{5s}{21}}}  \right).
\end{align*}
\end{lem}
\begin{proof} We consider the following cases.\\

\noindent\textbf{Case 1 : $p \neq 2$.} Since $(2,p)=1$, any solution $(n_1,n_2,\dots,n_s)$ of \eqref{f(x) formulas} is also a solution of
\begin{align}
\tilde{g}(n_1)+\tilde{g}(n_2)+\cdots+\tilde{g}(n_s)\equiv 2m \bmod p^k,\label{g(x) formula}
\end{align}
where $\tilde{g}(x) = 5x^3-5x^2+2x$. Therefore, we can write
\begin{align}
\mathcal{M}_m(p^k)&=\frac{1}{p^k}\sum_{t=1}^{p^k}\sum_{n_1=1}^{p^k}\sum_{n_2=1}^{p^k}\cdots\sum_{n_s=1}^{p^k}e(t(\tilde{g}(n_1)+\tilde{g}(n_2)+\cdots+\tilde{g}(n_s)-2m)/p^k)\notag \\
    &=p^{(s-1)k}+\frac{1}{p^k}\sum_{t=1}^{p^k-1}e\bigg(-\frac{2mt}{p^k}\bigg)\bigg(\sum_{x=1}^{p^k}e\bigg(\frac{t\tilde{g}(x)}{p^k}\bigg)\bigg)^s. \label{Primes > 7 Step 1}
\end{align}
Each $1\leqslant t\leqslant p^k-1$ can be written uniquely as $t = bp^{k-r}$, where $(b,p)=1$, $1\leqslant r\leqslant k$, and $1\leqslant b\leqslant p^r$. Thus, the second term in the right-hand side of \eqref{Primes > 7 Step 1} can be rewritten as
\begin{align}
\frac{1}{p^k}\sum_{r=1}^k&\sum\limits_{\substack{b=1 \\ (b,p)=1}}^{p^r}e\bigg(\frac{-2mb}{p^r}\bigg) \bigg(\sum_{x=1}^{p^k}e\bigg(\frac{b\tilde{g}(x)}{p^r}\bigg)\bigg)^s \notag \\
&=\frac{1}{p^k}\sum_{r=1}^k\sum\limits_{\substack{b=1 \\ (b,p)=1}}^{p^r}e\bigg(\frac{-2mb}{p^r}\bigg)\bigg(p^{k-r}\sum_{x=1}^{p^r}e\bigg(\frac{b\tilde{g}(x)}{p^r}\bigg)\bigg)^s \notag \\
&=p^{(s-1)k}\sum_{r=1}^k p^{-rs}\sum\limits_{\substack{b=1 \\ (b,p)=1}}^{p^r}e\bigg(\frac{-2mb}{p^r}\bigg) \bigg(\sum_{x=1}^{p^r}e\bigg(\frac{b\tilde{g}(x)}{p^r}\bigg)\bigg)^s.\label{Primes > 7 Step 2}
\end{align}
Combining \eqref{Primes > 7 Step 1} and \eqref{Primes > 7 Step 2}, we obtain
\begin{align*}
\left|\mathcal{M}_m(p^k)-p^{(s-1)k}\right| \leqslant p^{(s-1)k}\sum_{r=1}^k p^{-rs}\sum\limits_{\substack{b=1 \\ (b,p)=1}}^{p^r} \left|\sum_{x=1}^{p^r}e\bigg(\frac{b\tilde{g}(x)}{p^r}\bigg)\right|^s.
\end{align*}
Applying Lemma \ref{Weyl Inequality V2} with $X=q=p^r$, we have
\begin{align}
\left|\sum_{x=1}^{p^r}e\bigg(\frac{b\tilde{g}(x)}{p^r}\bigg)\right|&\leqslant 2p^{\frac{3r}{4}}+8p^{r+\frac{0.53305r}{\log (2r \log p)}}(2p^{-r}+p^{-2r})^{\frac{1}{4}}(\text{log }p^r)^{\frac{1}{4}}\notag\\
&\leqslant 2p^{\frac{3r}{4}}+8p^{r+\frac{0.53305r}{\log (2r \log p)}}(3p^{-r})^{\frac{1}{4}}(\text{log }p^r)^{\frac{1}{4}}\notag\\
&\leqslant 2p^{\frac{3r}{4}}+12p^{\frac{3r}{4}+\frac{0.53305r}{\log (2r \log p)}+\frac{\log (r \log p)}{4\log p}}.\label{3.7 weyl bound application}
\end{align}
For any $r\geqslant 1$ and $p \geqslant e^{e^{45}}$, we have
\begin{align}\label{finding p values}
\frac{0.53305r}{\log (2r \log p)}+\frac{\log (r \log p)}{4\log p} \leqslant \frac{r}{84}.
\end{align}
One may check that \eqref{finding p values} also holds for all $p \geqslant 3$ if $r\geqslant e^{45}$. Furthermore, for $3 \leqslant p < e^{e^{45}}$ and $1 \leqslant r \leqslant e^{45}$, trivially we have
\[\left|\sum_{x=1}^{p^r}e\bigg(\frac{b\tilde{g}(x)}{p^r}\bigg)\right| \leqslant e^{e^{89}} p^{\frac{16r}{21}}.
\]
Thus substituting the bound \eqref{finding p values} in \eqref{3.7 weyl bound application}, we see that for any odd prime $p$ and any $r\in\mathbb{N}$,
\begin{align} \left|\sum_{x=1}^{p^r}e\bigg(\frac{b\tilde{g}(x)}{p^r}\bigg)\right|& \leqslant e^{e^{89}} p^{\frac{16r}{21}}.\label{p>=7 Weyl bound}
\end{align}
\noindent\textbf{Case 2 : $p=2$.} Using the same idea as in Case 1, we have
\begin{align*}
\mathcal{M}_m(2^k)
    &=2^{(s-1)k}+\frac{1}{2^k}\sum_{t=1}^{2^k-1}e\bigg(-\frac{mt}{2^k}\bigg)\bigg(\sum_{x=1}^{2^k}e\bigg(\frac{t\tilde{g}(x)}{2^{k+1}}\bigg)\bigg)^s.
\end{align*}
It follows that
\begin{align*}
\left|\mathcal{M}_m(2^k)-2^{(s-1)k}\right|& 
\leqslant 2^{(s-1)k}\sum_{r=1}^k 2^{-rs}\sum\limits_{\substack{b=1 \\ (b,2)=1}}^{2^r} \left|\sum_{x=1}^{2^r}e\bigg(\frac{b\tilde{g}(x)}{2^{r+1}}\bigg)\right|^s.
\end{align*}
Applying Lemma \ref{Weyl Inequality V2} with $X=2^r$ and $q=2^{r+1}$, we have for any $r \in \mathbb{N}$,
\begin{align}
\left|\sum_{x=1}^{2^r}e\bigg(\frac{b\tilde{g}(x)}{2^{r+1}}\bigg)\right|&\leqslant 2^{\frac{3r}{4}+1}+8\cdot 2^{r+\frac{0.53305r}{\log(r \log 4)}}(2^{-r}+2^{-(r+1)}+2^{-2r+1})^{\frac{1}{4}}(\text{log }2^{r+1})^{\frac{1}{4}}\notag\\
&\leqslant 2\cdot2^{\frac{3r}{4}}+11\cdot 2^{\frac{3r}{4}+\frac{0.53305r}{\log(r \log 4)}+\frac{\log((r+1)\log 2)}{4 \log 2}} \leqslant 2^{e^{44}}\cdot 2^{\frac{16r}{21}}.\label{p=2 Weyl bound}
\end{align}
Combining \eqref{T Definition}, \eqref{Primes > 7 Step 1} and \eqref{Primes > 7 Step 2}, for $p$ odd, we write
\begin{align}    T_m(p)=1+\sum_{r=1}^\infty p^{-rs}\sum\limits_{\substack{b=1 \\ (b,p)=1}}^{p^r}e\bigg(\frac{-2mb}{p^r}\bigg)\bigg(\sum_{x=1}^{p^r}e\bigg(\frac{b\tilde{g}(x)}{p^r}\bigg)\bigg)^s. \notag
\end{align}
Furthermore, we have
\begin{align*}    
T_m(2)=1+\sum_{r=1}^\infty 2^{-rs}\sum\limits_{\substack{b=1 \\ (b,2)=1}}^{2^r}e\bigg(\frac{-mb}{2^r}\bigg)\bigg(\sum_{x=1}^{2^r}e\bigg(\frac{b\tilde{g}(x)}{2^{r+1}}\bigg)\bigg)^s .
\end{align*}
Putting together \eqref{p>=7 Weyl bound} and \eqref{p=2 Weyl bound}, along with the above two series representations for $T_m(p)$, we deduce that for $s \geqslant 9$ and any prime $p$,
\begin{align}
\left|T_m(p)-1\right|&\leqslant\sum_{r=1}^\infty p^{-rs}\sum\limits_{\substack{b=1 \\ (b,p)=1}}^{p^r} e^{se^{89}}p^{\frac{16rs}{21}}\notag 
= e^{se^{89}}\bigg(\frac{p - 1}{p}\bigg)\cdot \bigg( \frac{p^{1 - \frac{5s}{21}}}{1 - p^{1 - \frac{5s}{21}}} \bigg) \notag,
\end{align}
which concludes the proof.
\end{proof}
\subsection{Hasse--Weil Bounds and Hensel Lifting}\label{subsec: Hensel Lifting} Consider \eqref{f(x) formulas} in the form
\begin{align}\label{f(x) formulas mod p}
g(n_1)+g(n_2)+\cdots+g(n_s)\equiv m \bmod p,
\end{align}
where $p \geqslant 7$ prime and $1 \leqslant n_i \leqslant p$ for $i=1,2,\dots, s$. Fix $n_{3}, \dots, n_{s}$ and set 
\begin{align}
x: = n_1, \quad y :=n_2, \quad B := g(n_{3}) + \cdots + g(n_{s}) - m.\label{definition of B}
\end{align}
Since $(2,p)=1$, \eqref{f(x) formulas mod p} can be rewritten as
\begin{align}
5x^{3} - 5x^{2} + 2x + 5y^{3} - 5y^{2} + 2y + 2B \equiv 0 \bmod p.\label{elliptic curve}
\end{align}
\begin{lem}\label{elliptic curve satisfied}
Suppose $p \geqslant 7$ and let $\mathcal{N}_{p}$ denote the number of solutions to \eqref{elliptic curve} over $\mathbb{F}_{p}$. Then 
\[-2\sqrt{p}\leqslant\mathcal{N}_{p}-(p+1)\leqslant p.
\]
\end{lem}
\begin{proof}
Let $\frac{1}{a}$ denote the inverse of $a\bmod p$. Since $p\geqslant 7$, \eqref{elliptic curve} can be reduced to 
\begin{align}
   x^3-x^2+\frac{2}{5}x+y^3-y^2+\frac{2}{5}y+\frac{2}{5}B\equiv 0\bmod p.\label{6.6.1} 
\end{align}
Let $R_1=x+y-\frac{2}{3}$, $ R_2=x-y$ and $B_0=\frac{64}{135}+\frac{8B}{5}$. Then we can rewrite \eqref{6.6.1} as
\begin{align}
    R_1^3+3R_1R_2^2+\frac{4}{15}R_1+B_0\equiv 0\bmod p.\label{6.6.3} 
\end{align}
\indent We consider the following cases.\\

\noindent \textbf{Case 1 : $B\equiv -\frac{8}{27}\bmod p$.}
Then $B_0\equiv 0\bmod p$. Hence \eqref{6.6.3} is reduced to 
\begin{align}
    R_1  (R_1^2+3R_2^2+\frac{4}{15} )\equiv 0\bmod p.\label{reduced 6.6.3}
\end{align}
All solutions to \eqref{reduced 6.6.3} come from either 
\begin{align}
    R_1&\equiv 0\bmod p, \quad \textrm{or}\label{case 1, subcase 1}\\
  R_1^2+3R_2^2+\frac{4}{15}&\equiv 0\bmod p. \label{case 1, subcase 2}
\end{align} 
If $R_1\equiv 0\bmod p$, then $x+y\equiv \frac{2}{3}\bmod p$. Thus, the number of solutions to \eqref{case 1, subcase 1} is equal to $p$. If $R_1^2+3R_2^2+\frac{4}{15}\equiv 0\bmod p$, then since $h(R_1, R_2) = R_1^2+3R_2^2+\frac{4}{15}$ is an irreducible non-singular curve over $\mathbb{F}_p$, the number of solutions to \eqref{case 1, subcase 2} is $p+1$. Combining the two cases, there are at most $2p+1$ solutions to \eqref{reduced 6.6.3}. If both \eqref{case 1, subcase 1} and \eqref{case 1, subcase 2} hold, then 
\[
3R_2^2+\frac{4}{15}\equiv 0 \bmod p,
\]
which has at most 2 solutions. Thus, the total number of solutions to \eqref{reduced 6.6.3} is between $2p-1$ and $2p+1$. In particular, this implies that $2p-1\leqslant\mathcal{N}_p \leqslant 2p+1$. \\

\noindent \textbf{Case 2 : $B\not\equiv -\frac{8}{27}\bmod p$.} Since $B\not\equiv -\frac{8}{27}\bmod p$, we have $B_0\not\equiv 0\bmod p$. If $R_1\equiv 0\bmod p$, then \eqref{6.6.3} is reduced to $B_0\equiv 0\bmod p$, which is a contradiction. Therefore, multiplying \eqref{6.6.3} by $R_1^{-3}$, and writing $S_1=R_2R_1^{-1}, S_2=R_1^{-1}$, we get
\begin{align}
1+3S_1^2+\frac{4}{15}S_2^2+\bigg (\frac{64}{135}+\frac{8B}{5}\bigg )S_2^3&\equiv 0 \bmod p \label{6.6.5}.
\end{align}
We then multiply \eqref{6.6.5} by $27B_0^2$, and let $T_1=9B_0S_1$ and $T_2=-3B_0S_2$ to obtain
\begin{align}
    T_1^2\equiv T_2^3-\frac{4}{5}T_2^2-27B_0^2\bmod p.\label{6.6.6}
\end{align}
Finally, we let $Z_1=T_1, Z_2=T_2-\frac{4}{15}$. Then \eqref{6.6.6} becomes 
\begin{align}
    Z_1^2\equiv Z_2^3+\frac{4}{225}Z_2+\bigg (-\frac{128}{3375}-27B_0^2 \bigg)\bmod p,\label{6.22}
\end{align}
which is an elliptic curve when 
\begin{align}
    \frac{4^4}{15^6}+27\bigg (\frac{128}{3375}-27B_0^2 \bigg)^2\not\equiv 0\bmod p.\label{discriminant}
\end{align}
Note that \eqref{discriminant} is not satisfied if and only if 
\begin{align}
    j(Z_2):= Z_2^3+\frac{4}{225}Z_2+\bigg(-\frac{128}{3375}-27B_0^2\bigg)\bmod p\label{j(z)}
\end{align}
has at least two equal roots in the algebraic closure $\overline{\mathbb{F}_p}$ of $\mathbb{F}_p$. Suppose all three roots to \eqref{j(z)} are equal, denoted by $\alpha$. Then 
\begin{align}
    j(Z_2) = (Z_2-\alpha)^3\label{j(z) three equal roots}
\end{align}
Equating the coefficients of \eqref{j(z) three equal roots} and the right-hand side of \eqref{6.22}, we have $\frac{4}{225}\equiv 0\bmod p$, which is a contradiction since $p\geqslant 7$. Now suppose $j(Z_2)$ has two roots, say $\alpha,\beta$, in the algebraic closure $\overline{\mathbb{F}_p}$ with $\alpha \neq \beta$ and $j(Z_2)=(Z_2-\alpha)^2(Z_2-\beta)$. Again, equating the coefficients, we obtain the system of equations
\begin{align}
    j(Z_2)& \equiv (Z_2-\alpha)^2(Z_2+2\alpha) \bmod p,\notag\\
    \frac{4}{225}&\equiv-3\alpha^2 \bmod p,\notag\\
    \textrm{and} \quad -\frac{128}{3375}-27B_0^2 &\equiv 2\alpha^3 \bmod p.\label{2 alpha cubed}
\end{align}
From the above equation, both $\alpha^2$ and $\alpha^3$ are in $\mathbb{F}_p$. Moreover, $\alpha^2$ is invertible. This forces $\alpha$ to be in $\mathbb{F}_p$. Now \eqref{6.22} is equivalent to
\begin{align}
Z_1^2 \equiv (Z_2-\alpha)^2(Z_2+2\alpha)\bmod p, \label{6.21}
\end{align}
where $\alpha \equiv \frac{2}{45}\sqrt{-3}\bmod p$ is in $\mathbb{F}_p$. The number of solutions to \eqref{6.21} comes from either $Z_2\equiv\alpha\bmod p$ or $Z_2\not\equiv\alpha\bmod p$. In the first case, when $Z_2\equiv\alpha\bmod p$, it forces $Z_1\equiv 0\bmod p$, which implies that $x=y$. Substituting $x=y$ in all variables, since $Z_2=T_2-4/15$, we get
\begin{align}
    \alpha= T_2-\frac{4}{15} = -3B_0S_2-\frac{4}{15} = \frac{-9B_0}{6x-2}.\label{unique solution for x=y}
\end{align}
Note that since $R_1\not\equiv 0\bmod p$, $6x-2\not\equiv 0\bmod p$. Therefore, \eqref{unique solution for x=y} is equivalent to $x = -\frac{3B_0}{2\alpha}+\frac{1}{3}$. This means that $x=y$ has a unique solution in $\mathbb{F}_p$. In the latter case, when $Z_2\not\equiv\alpha\bmod p$, let $t=Z_1 (Z_2-\alpha)^{-1}$. Then \eqref{6.21} is equivalent to 
\begin{align}
    t^2\equiv Z_2+2\alpha\bmod p.\label{parabola equation}
\end{align}
This is a nonsingular parabola, so the number of solutions to \eqref{parabola equation} is $p+1$. Adding the number of solutions of both situations for multiple roots, we conclude that when \eqref{6.22} is not an elliptic curve, $\mathcal{N}_p = p+2$. The only remaining case is when \eqref{6.22} is an equation of an elliptic curve. For this case, we apply the Hasse--Weil Theorem, established in a series of papers (see \cite{hassecongruentfunctions},  \cite{hassecomplexmultiplication}, \cite{hassefiniteorder}, and \cite{weil1949}), and conclude that $\lvert\mathcal{N}_{p}-(p+1)\rvert\leqslant 2\sqrt{p}$.
\end{proof}
Given a prime $p \geqslant 7$, consider the congruence equation
\begin{align}\label{Derivative Equation}
\frac{15}{2}x^2-5x+1\equiv 0 \bmod p.
\end{align}
There are at most 2 solutions $\text{mod }p$ to \eqref{Derivative Equation}, say $\alpha$ and $\beta$. Given a solution $(\delta_1,\delta_2, \dots, \delta_s)$ to \eqref{f(x) formulas mod p}, if there exists at least some $j \in \{1,2,\dots,s\}$ such that $\delta_j\not\in\{\alpha,\beta\}$, we call such a solution a ``good" solution. We have the following lemma.
\begin{lem}\label{hensel's lemma lifting}
Let $p \geqslant 7$ be a prime, $k \geqslant 2$ and consider the congruence equation
\begin{align}\label{f(x) formulas Hensel Lifting}
g(n_1)+g(n_2)+\cdots+g(n_s)\equiv m \bmod p^k,
\end{align}
where $g(x)=\frac{5}{2}x^3-\frac{5}{2}x^2+x$ and $1\leqslant n_i\leqslant p^k$ for all $i \in \{1,2,\dots,s\}$. Then every good solution of \eqref{f(x) formulas mod p} can be lifted uniquely to $p^{(k-1)(s-1)}$ solutions of \eqref{f(x) formulas Hensel Lifting}. 
\end{lem}
\begin{proof}
Let $(\delta_1,\delta_2,\dots,\delta_s)$ be a good solution to \eqref{f(x) formulas mod p}. Then there exists some $\delta_j\not\in\{\alpha,\beta\}$ which implies that $g'(\delta_j)\not\equiv 0 \bmod p$. By Hensel's lemma, $\delta_j$ can be extended uniquely to a solution $\text{mod }p^k$ for any $k\geqslant 2$. In other words, for any $n_1,\dots,n_{j-1},n_{j+1},\dots,n_s$ such that $n_i\equiv \delta_i \bmod p$, we have
\begin{align*}
g(n_1)+\cdots+g(n_{j-1})+g(\delta_j)+g(n_{j+1})+g(n_{s})\equiv m \bmod p^k.
\end{align*}
Hence in this case, $(\delta_1,\delta_2,\dots,\delta_s)$ can be lifted to $p^{(k-1)(s-1)}$ solutions mod $p^k$.
\end{proof}
\begin{rem}\label{Hensel Remark}
Suppose in \eqref{f(x) formulas mod p}, we fix all but $r$ of the variables $n_i$'s. Then the number of solutions to \eqref{f(x) formulas mod p} which are not good is at most $2^r$.
\end{rem}
\begin{lem}\label{T_m(p) lower bound}
For $s\geqslant 9$ and any prime $p$, we have
\[T_m(p)> \max\bigg\{p^{1-s},1-\frac{2}{\sqrt{p}}-\frac{3}{p}\bigg\}.\]
\end{lem}
\begin{proof}
We consider the following cases. \\

\noindent \textbf{Case 1 : $p\geqslant 11$.} For $p\geqslant 11$, by Lemma \ref{elliptic curve satisfied}, the number of solutions $\mathcal{N}_{p}$ to \eqref{elliptic curve} satisfies $\mathcal{N}_{p}-(p+1)\geqslant -2\sqrt{p}.$ It follows that
\begin{align}\label{final hasse bound}
\mathcal{M}_{m}(p) -(p+1)p^{s-2} \geqslant  -2\sum^{p}_{n_{3} = 1} \sum^{p}_{n_{4} = 1} \cdots \sum^{p}_{n_{s} = 1}  \sqrt{p} \geqslant -2p^{s-\frac{3}{2}}.
\end{align}
From Remark \ref{Hensel Remark}, there are at most 4 solutions to \eqref{elliptic curve} such that Hensel's Lemma cannot be applied when lifting the solution $\bmod \hspace{0.1cm}p^k$. When $p\geqslant 11$, $\mathcal{N}_{p}> 4$. So there exists at least one solution to which Hensel's Lemma can be applied. Thus for $p\geqslant 11$, by Lemma \ref{hensel's lemma lifting} we have
\begin{align*}
\mathcal{M}_m(p^k) \geqslant (\mathcal{M}_m(p)-4p^{s-2})p^{(k - 1)(s - 1)}.
\end{align*}
We deduce that
\begin{align}\label{p>11 Case 1}
T_m(p) \geqslant \lim_{k \rightarrow \infty} \frac{(\mathcal{M}_{m}(p)-4p^{s-2}) p^{(k - 1)(s-1)}}{p^{k(s-1)}}= \frac{\mathcal{M}_{m}(p)-4p^{s-2}}{p^{s-1}}.
\end{align}
\noindent Substituting \eqref{final hasse bound} in \eqref{p>11 Case 1}, we obtain
\begin{align*}
T_m(p)\geqslant 1-\frac{2}{\sqrt{p}}-\frac{3}{p}.
\end{align*}
\textbf{Case 2 : $p=2$.} Consider the congruence equation
\begin{align}\label{1 Variable}
g_0(x):=\frac{5x^3}{2} - \frac{5x^2}{2} + x + B \equiv 0 \bmod 2,
\end{align}
where $B = g(n_{2}) + \cdots + g(n_{s}) - m$. Note $g(x)$ and $x$ have the same parity. Hence, for any $B$, $x\equiv B \bmod 2$ is always a solution to \eqref{1 Variable}. Also, when $x\equiv 0 \bmod 2$, $g_0'(x) \not\equiv 0 \bmod 2$. Thus, for any even $B$, Hensel's Lemma can be applied. Since we can arrange the parity of $g(n_2)+\cdots+g(n_s)$, such an even $B$ always exists for any $m$. Applying Lemma \ref{hensel's lemma lifting}, we have $\mathcal{M}_m(2^k)\geqslant 2^{(k-1)(s-1)},$
which implies that
\begin{align*}
T_m(2) \geqslant \lim_{k \rightarrow \infty}2^{k(1-s)}2^{(k-1)(s-1)}=2^{1-s}>0.
\end{align*}
\textbf{Case 3 : $p=3,5,7$.} It suffices to consider the congruence equation
\begin{align}\label{1 Variable mod 3}
g_0(x):=\frac{5x^3}{2} - \frac{5x^2}{2} + x + B \equiv 0 \bmod 3,
\end{align}
where $B = g(n_{2}) + \cdots + g(n_{s}) - m$. When $x\equiv 0,1 \bmod 3$, $g_0'(x)\not\equiv 0 \bmod 3$. Therefore, $B$ must satisfy $\equiv 0,2 \bmod 3$. If $m\equiv 0 \bmod 3$, then we let all $n_i\equiv 0 \bmod 3$ for $i=2,3,\dots, s$ which ensures that $B\equiv 0 \bmod 3$. If $m\equiv 1 \bmod 3$, then choose $n_2\equiv 1 \bmod 3$ and the rest of $n_i\equiv 0 \bmod 3$ so that $B\equiv 2$ (mod $3$). Finally, if $m\equiv 2 \bmod 3$, then we pick $n_2, n_3\equiv 1 \bmod 3$ and the rest of $n_i\equiv 0 \bmod 3$ so that $B\equiv 2 \bmod 3$. Therefore, there always exists a solution of $\eqref{1 Variable mod 3}$ to which Hensel's lemma can be applied. Applying Lemma \ref{hensel's lemma lifting}, we obtain $\mathcal{M}_m(3^k)\geqslant 3^{(k-1)(s-1)},$
which implies $T_m(3)\geqslant 3^{1-s}>0.$ The arguments for $p=5,7$ are similar.
\end{proof}
\subsection{A lower bound for $\mathfrak{S}(m)$}
We are ready to establish a lower bound for $\mathfrak{S}(m)$.
\begin{lem}\label{lower bound for S(m)}
We have $\mathfrak{S}(m)>0$. More precisely,
\begin{align}\label{Sigma Bound}
 \mathfrak{S}(m) \geqslant \num{1.56e-3}\cdot210^{1-s}\cdot 2^{\frac{42z}{42-5s}} \exp\bigg(-\frac{z \log 2}{\log z}\bigg(1+\frac{1.2762}{\log z}\bigg)\bigg),
\end{align}
where $z=(2 e^{se^{89}}+1)^{\frac{21}{5s - 21}}$.
\end{lem}
\begin{proof}
Define $\theta_p := T_m(p)-1$ for $p$ prime and $w = \prod_{p > z} T_{m}(p)^{-1}$. From Lemma \ref{T_m(p) complete bound}, it follows that when  $p > z$, $\left|\theta_p\right|\leqslant\frac{1}{2}$.
Using Taylor expansion, we write 
\begin{align}\label{Taylor Expansion}
    \log w  
    &= - \sum_{p > z} \log(1 + \theta_{p})=\sum_{p > z} \left(\theta_{p} + \frac{\theta^{2}_{p}}{2} + \frac{\theta^{3}_{p}}{3} + \cdots \right).
\end{align}
Hence by Lemma \ref{T_m(p) complete bound} and \eqref{Taylor Expansion}, we have
\begin{align*}
    \log w &\leqslant 4 \log 2 \cdot e^{se^{89}}\sum_{p > z} p^{1-\frac{5s}{21}} \leqslant 4 \log 2 \cdot e^{se^{89}} \int^{\infty}_{z} x^{1-\frac{5s}{21}}  \dd x
= \frac{42z\log 2}{5s-42}.
\end{align*}
Exponentiating both sides of the above inequality gives $ w \leqslant  2^{\frac{42z}{5s-42}},$
which implies that
\begin{align}
\prod_{p > z} T_{m}(p)\geqslant 2^{\frac{42z}{42-5s}} .  \label{Tm(p) p>M}  
\end{align}
For $p\leqslant z$, Lemma \ref{T_m(p) lower bound} shows that
\begin{align}
\prod_{p \leqslant z} T_m(p) &\geqslant\prod_{p \leqslant z} \max \bigg\{p^{1-s},1-\frac{2}{\sqrt{p}}-\frac{3}{p}\bigg\}\notag\\
& \geqslant 210^{1-s}\exp\left(\sum_{11\leqslant p\leqslant z}\log\left(1-\frac{2}{\sqrt{p}}-\frac{3}{p}\right)\right)\notag\\
&\geqslant \num{1.56e-3}\cdot210^{1-s}\exp\left(-\sum_{29\leqslant p\leqslant z}\sum_{j=1}^\infty\frac{1}{j}\left(\frac{3+2\sqrt{p}}{p}\right)^j\right)\notag\\
&\geqslant\num{1.56e-3}\cdot 210^{1-s}\exp\left(-\frac{z \log 2}{\log z}\left(1+\frac{1.2762}{\log z}\right)\right) \label{Tm(p) p<= M},
\end{align}
where the last inequality follows from Dusart's upper bound for $\pi(x)$ (see \cite[Theorem 6.9]{dusart2010estimates}). Combining \eqref{Tm(p) p>M} and \eqref{Tm(p) p<= M}, we have
\begin{align*}
     \mathfrak{S}(m)&\geqslant \num{1.56e-3}\cdot210^{1-s}\cdot 2^{\frac{42z}{42-5s}} \exp\bigg(-\frac{z \log 2}{\log z}\bigg(1+\frac{1.2762}{\log z}\bigg)\bigg) >0,
\end{align*}
which completes the proof.
\end{proof}
\section{Major Arcs : The Singular Integral}\label{sec: Major Arcs Sec III}  Here we estimate the singular integral $J^{*}(m)$ given by \eqref{J* definition}. We assume $N^{\delta} \geqslant e^{e^{45}}$.
\subsection{Preliminaries} Let $N_0 =\frac{5}{2}N^3$ where $N$ is given by \eqref{Defining N}. Define
\begin{align}
    v_1(\theta):= \frac{1}{3}\sum_{1 \leqslant n \leqslant N_0} n^{-2/3}e(\theta n) \quad \textrm{and} \quad v_2(\theta) := \int_0^{N_0^{1/3}} e(\theta t^3) \dd t.\label{v_1 definition}
\end{align}
\begin{lem}\label{v_3(theta)-v_1(theta) bound lemma}
Let $v(\theta)$ and $v_1(\theta)$ be as defined in \eqref{v theta definition} and \eqref{v_1 definition} respectively. Then
\[\bigg \lvert v_1(\theta)-\left(\frac{5}{2}\right)^{1/3}v(\theta)\bigg \rvert \leqslant 48N^\delta.
\]
\end{lem}
\begin{proof}We have
\begin{align}
   \bigg \lvert v_2(\theta) -\left(\frac{5}{2}\right)^{1/3}v(\theta) \bigg \rvert &\leqslant \int_0^{\sqrt[3]{\frac{5}{2}}} \left|e(\theta t^3)\right| \dd t=\left ( \frac{5}{2}\right)^{1/3}\label{|v_2(theta)-v_1(theta)|}.
\end{align}
Since $f(n) = n^{-2/3}$ is an decreasing function,
\begin{align*}
\int_1^{N_0+1}t^{-2/3}\dd t \leqslant \sum_{1 \leqslant n \leqslant N_0}n^{-2/3}\leqslant \int_0^{N_0}t^{-2/3} \dd t.
\end{align*}
Therefore, we get
\begin{align*}
\left|\frac{1}{3}\sum_{1 \leqslant n \leqslant N_0}n^{-2/3}-\frac{1}{3}\int_1^{N_0}t^{-2/3} \dd t\right| \leqslant \frac{2}{3} \left|\int_0^1 t^{-2/3} \dd t\right|= 2.
\end{align*}
Since $\frac{1}{3}\int_1^{N_0}t^{-2/3}dt=N_0^{1/3}-1,$ we obtain
\begin{align}
\bigg|\frac{1}{3}\sum_{1 \leqslant n \leqslant N_0} n^{-2/3}-N_0^{1/3}\bigg|\leqslant 3.\label{intermediate step}
\end{align}
By partial summation,
\begin{align}\label{Partial Summation v_1}
    v_1(\theta)=e(N_0\theta)\bigg(\frac{1}{3}\sum_{1 \leqslant n \leqslant N_0} n^{-2/3}\bigg)-2\pi i\theta\int_1^{N_0} \bigg(\frac{1}{3}\sum_{1 \leqslant n \leqslant t} n^{-2/3}\bigg)e(\theta t) \dd t.
\end{align}
On the other hand, using integration by parts and change of variables, we obtain
\begin{align}\label{By Parts v_2}
v_2(\theta) = e(N_0\theta)N_0^{1/3}-2\pi i\theta\int_0^{N_0} t^{1/3} e(\theta t) \dd t.
\end{align} 
Combining \eqref{intermediate step}, \eqref{Partial Summation v_1} and \eqref{By Parts v_2}, we deduce that
\begin{align}
\left|v_1(\theta)-v_2(\theta) \right| &\leqslant \bigg \lvert \frac{1}{3}\sum_{1 \leqslant n \leqslant N_0} n^{-2/3}-N_0^{1/3}\bigg \rvert +2\pi|\theta|\bigg (1+ \int_1^{N_0} \bigg \lvert t^{1/3} - \bigg(\frac{1}{3}\sum_{1 \leqslant n \leqslant t} n^{-2/3}\bigg)\bigg \rvert \dd t  \bigg) \notag\\
&\leqslant3+2\pi|\theta|(3N_0-2) \label{v_1-v_2}.
\end{align}
Recall that $|\theta|<N^{\delta-3}$. Putting together \eqref{|v_2(theta)-v_1(theta)|} and \eqref{v_1-v_2}, we have
\begin{align*}
\bigg \lvert v_1(\theta)-\left(\frac{5}{2}\right)^{1/3}v(\theta)\bigg \rvert \leqslant  3+\pi(6N_0-4)N^{\delta-3}+\sqrt[3]{5/2} \leqslant 48N^\delta,
\end{align*}
which completes the proof.
\end{proof}
\begin{lem}\label{V_1 theta Estimation}
Let $|\theta| \leqslant \frac{1}{2}$ and $v_1(\theta)$ be as in \eqref{v_1 definition}. Then 
\[
|v_1(\theta)|\leqslant \min \{2m^{1/3}, 2 |\theta|^{-1/3} \}.
\]
\end{lem}
\begin{proof}
If $|\theta|\leqslant m^{-1}$, the trivial bound suffices. Assume $|\theta| > m^{-1}$ and $M=[ |\theta|^{-1}]$. Then the contribution to $v_1(\theta)$ from the terms $n\leqslant M$ is $\leqslant M^{1/3}\leqslant |\theta|^{-1/3}$. When $n> M$, let $S_n = \sum_{1 \leqslant r \leqslant n}e(\theta r)$ and $c_n = \frac{1}{3}n^{-2/3}$. 
Then we can write
\begin{align*}
 \frac{1}{3}\sum_{n=M+1}^{N_0} n^{-2/3}e(\theta n)=c_{N_0+1}S_{N_0}-c_{M+1}S_M+\sum_{n=M+1}^{N_0}(c_n-c_{n+1})S_n.
\end{align*}
By Lemma \ref{Geometric Sum}, $\lvert S_n\rvert\leqslant\frac{1}{2\lvert\theta\rvert}$. Since $c_n$ is strictly decreasing, we have
\begin{align*}
\frac{1}{3}\sum_{n=M+1}^{N_0} n^{-2/3}e(\theta n)&= -c_{M+1}S_M+\sum_{n=M+1}^{N_0-1}(c_n-c_{n+1})S_n+c_{N_0}S_{N_0}\\
&\leqslant c_{M+1}S_M+\sum_{n=M+1}^{N_0-1}(c_n-c_{n+1})\frac{1}{2\lvert\theta\rvert}+c_{N_0}\frac{1}{2\lvert\theta\rvert}\\
&\leqslant 2c_{M+1}\frac{1}{2\lvert\theta\rvert}\leqslant \lvert\theta\rvert^{-1/3}.
\end{align*} 
Combining the two parts, $|v_1(\theta)|\leqslant 2|\theta|^{-1/3}$ and the desired conclusion follows.
\end{proof}
\begin{lem}\label{gamma approximation}
    Suppose $\alpha,\beta$ are real numbers with $\alpha \geqslant \beta>0$ and $\beta<1$. Then
    \begin{align*}
        \left|\sum_{n=1}^{m-1} n^{\beta-1}(m-n)^{\alpha-1}-m^{\beta+\alpha-1}\bigg(\frac{\Gamma(\beta)\Gamma(\alpha)}{\Gamma(\beta+\alpha)}\bigg)\right|\leqslant \frac{2}{\beta}m^{\alpha-1}{}_2F_1\left(\beta,1-\alpha,1+\beta,\frac{1}{m} \right),
    \end{align*}
    where $\Gamma$ is the Gamma function and ${}_2F_1$ is the hypergeometric function. When $\beta=\frac{1}{3}$,     
\[  \left|\sum_{n=1}^{m-1} n^{\beta-1}(m-n)^{\alpha-1}-m^{\beta+\alpha-1}\bigg(\frac{\Gamma(\beta)\Gamma(\alpha)}{\Gamma(\beta+\alpha)}\bigg)\right| \leqslant 12m^{\alpha-1}.
\]
\end{lem}
\begin{proof}
Consider the function $\phi(t)=t^{\beta-1}(m-t)^{\alpha-1}$ in the interval $(0,m)$. Since $\phi$ has at most one stationary point, we can divide $(0,m)$ into two intervals $(0, Y]$ and $(Y,m)$ (one of which might be empty) such that $\phi$ is decreasing on $(0, Y]$ and increasing on $(Y,m)$. Then
\begin{align}\label{Phi Integral 1}
    \left|\sum_{n=1}^{m-1}\phi(n)-\int_0^m\phi(t)\dd t\right|
    &\leqslant \int_0^{1}\phi(t) \dd t+\int_{m-1}^{m}\phi(t)\dd t \leqslant 2\int_{0}^{1}t^{\beta-1}(m-t)^{\alpha-1}\dd t.
\end{align}
Let $B$ be the incomplete Beta function. We have
\begin{align}
   \int_{0}^{1}t^{\beta-1}(m-t)^{\alpha-1}\dd t
&=m^{\alpha+\beta-1}\int_{0}^{1/m}u^{\beta-1}(1-u)^{\alpha-1}\dd u =m^{\alpha+\beta-1} B\left (\frac{1}{m};\beta,\alpha\right ). \label{Phi Integral 2}
\end{align}
Note, $B(\frac{1}{m};\beta,\alpha)=\frac{1}{\beta}m^{-\beta} {}_2F_1(\beta,1-\alpha,1+\beta,\frac{1}{m})$ 
where ${}_2F_1$ is the hypergeometric function and is finite because $1+\beta$ is not a nonpositive integer. Hence combining \eqref{Phi Integral 1} and \eqref{Phi Integral 2}, 
\begin{align}\label{Phi Integral 4}
\left|\sum_{n=1}^{m-1}\phi(n)-\int_0^m\phi(t)\dd t\right|\leqslant \frac{2}{\beta}m^{\alpha-1}{}_2F_1\bigg (\beta,1-\alpha,1+\beta,\frac{1}{m}\bigg).
\end{align}
By the properties of the Gamma function, we have 
\begin{align*}
    \frac{\Gamma(\beta)\Gamma(\alpha)}{\Gamma(\beta+\alpha)}=\int_0^1t^{\beta-1}(1-t)^{\alpha-1} \dd t.
\end{align*}
It follows that
\begin{align}
 \int_0^m\phi(t) \dd t &= \int_0^m t^{\beta-1}(m-t)^{\alpha-1} \dd t =\frac{\Gamma(\beta)\Gamma(\alpha)}{\Gamma(\beta+\alpha)}m^{\beta+\alpha-1}\label{Phi Integral 5}.
\end{align}
Putting together \eqref{Phi Integral 1}, \eqref{Phi Integral 4}, and \eqref{Phi Integral 5}, the first assertion follows. In particular, when $\beta=\frac{1}{3}$, ${}_2F_1\left(\beta,1-\alpha,1+\beta,\frac{1}{m}\right) 
<2.$ 
Substituting this in \eqref{Phi Integral 4}, we obtain the second assertion.
\end{proof}
\subsection{Completing the Singular Integral} First, consider the integral
\begin{align}
J_1^*(m) &:= \int_{-N^{\delta-3}}^{N^{\delta-3}} v_1(\theta)^se(-\theta m) \dd \theta. \label{J_1* definition}
\end{align}
Therefore, by Lemma \ref{v_3(theta)-v_1(theta) bound lemma},
\begin{align}
|J_1^*(m)-J^*(m)|& \leqslant \int_{-N^{\delta-3}}^{N^{\delta-3}}\bigg \lvert v_1(\theta)^s-\left(\left(\frac{5}{2}\right)^{1/3}v(\theta)\right)^s\bigg \rvert \dd\theta\notag\\
&\leqslant 48N^{\delta}\int_{-N^{\delta-3}}^{N^{\delta-3}}\sum_{j=1}^s \lvert v_1(\theta) \rvert^{s-j}\cdot \bigg \lvert \left(\frac{5}{2}\right)^{1/3}v(\theta)\bigg \rvert^{j-1} \dd\theta\notag\\
&\leqslant 96sN^\delta N_0^{(s-1)/3}N^{\delta-3} \leqslant 72s\left(\frac{5}{2}\right)^{s/3}N^{s-4+2\delta}\label{J_1*(m)-J*(m)}.
\end{align}
\noindent Now, we extend the integral $J^{*}_1(m)$ to an integral over an unit interval. We define
\begin{align}
J_1(m) := \int_{-1/2}^{1/2}v_1(\theta)^se(-\theta m)\dd\theta.\label{J_1(m)}
\end{align}
We will approximate $J^{*}(m)$ with $J_1(m)$. By Lemma \ref{V_1 theta Estimation} and that $N^{\delta-3}< \frac{1}{2}$, we have
\begin{align}
|J_1(m)-J_1^*(m)| &\leqslant \bigg|\int_{-1/2}^{-N^{\delta-3}} v_1(\theta)^se(-\theta m)\dd\theta\bigg|+\bigg|\int_{N^{\delta-3}}^{1/2} v_1(\theta)^se(-\theta m)\dd\theta\bigg|\notag\\
    &\leqslant 2\int_{N^{\delta-3}}^{1/2}\bigg(\min \bigg \{2m^{1/3}, 2 |\theta|^{-1/3} \bigg \}\bigg)^s \dd\theta\notag\\
    &\leqslant 2^{s+1}\int_{N^{\delta-3}}^{1/2}\theta^{-s/3}\dd\theta \leqslant 2^{s+1}\frac{3}{s-3}N^{\frac{3\delta-s\delta+3s-9}{3}}.\label{J_1(m)-J_^*(m)}
\end{align}
Combining \eqref{J_1*(m)-J*(m)} with \eqref{J_1(m)-J_^*(m)}, it follows that
\begin{align}
|J_1(m)-J^*(m)| \leqslant 72s\left(\frac{5}{2}\right)^{s/3}N^{s-4+2\delta}+\frac{2^s \cdot 6}{s-3}N^{\frac{3\delta-s\delta+3s-9}{3}}\label{J_1(m)-J^*(m)}.
\end{align} 
We now prove the following lemma.
\begin{lem}\label{J_1(m) approximation bound}
Let $s\geqslant 2$. Then 
\begin{align*}
\bigg|J_1(m,s)-\Gamma\bigg(\frac{4}{3}\bigg)^s\Gamma\bigg(\frac{s}{3}\bigg)^{-1}m^{s/3-1}\bigg|\leqslant 10^{s-2}m^{(s-1)/3-1}.
\end{align*}

\end{lem}
\begin{proof}
We write
\begin{align*}
    J_1(m)=J_1(m,s) &= \int_{-1/2}^{1/2}v_1(\theta)^se(-\theta m) \dd\theta\\
&=3^{-s}\sum_{n_1=1}^{N_0}\cdots\sum_{n_s=1}^{N_0}(n_1n_2\cdots n_s)^{-2/3}\int_{-1/2}^{1/2}e(n_1+\cdots+n_s-m) \dd\theta\\
    &=3^{-s}\mathop{\sum_{n_1=1}^{N_0}\cdots\sum_{n_s=1}^{N_0}}_{n_1+\cdots+n_s=m} (n_1n_2\cdots n_s)^{-2/3}.
\end{align*}
When $s=2$, 
\begin{align*}
J_1(m,2)&=3^{-2}\mathop{\sum_{n_1=1}^{m-1}\sum_{n_2=1}^{m-1}}_{n_1+n_2=m}(n_1n_2)^{-2/3}=3^{-2}\sum_{n_1=1}^{m-1} n_1^{-2/3}(m-n_1)^{-2/3}.
\end{align*}
Applying Lemma \ref{gamma approximation} with $\alpha=\beta=\frac{1}{3}$, we deduce that
\begin{align*}
&\bigg|J_1(m,2)-3^{-2}m^{-1/3}\frac{\Gamma(\frac{1}{3})\Gamma(\frac{1}{3})}{\Gamma(\frac{2}{3})} \bigg|\leqslant \frac{4}{3}m^{-2/3}.
\end{align*}
Rewriting the above, we obtain
\[
\bigg|J_1(m,2)-\Gamma\bigg(\frac{4}{3}\bigg)^2\Gamma\bigg(\frac{2}{3}\bigg)^{-1}m^{-1/3}\bigg| \leqslant \frac{4}{3}m^{-2/3},
\]
and thus the lemma holds for $s=2$. Now, suppose the lemma holds for some $s\geqslant 2$. Note that 
\begin{align*}
    J_1(m,s+1)=\frac{1}{3}\sum_{n=1}^{m-1}n^{-2/3}J_1(m-n,s).
\end{align*}
By assumption,
\begin{align} \label{J1 bound Step 1}
\bigg|J_1(m-n,s)-\Gamma\bigg(\frac{4}{3}\bigg)^s\Gamma\bigg(\frac{s}{3}\bigg)^{-1}(m-n)^{s/3-1}\bigg|\leqslant 10^{s-2}(m-n)^{(s-1)/3-1}.
\end{align}
Moreover, by Lemma \ref{gamma approximation} with $\beta=1/3$ and $\alpha=s/3$,
\begin{align}\label{J1 bound Step 2}
\bigg|\frac{1}{3} \sum_{n=1}^{m-1}n^{-2/3}\Gamma\bigg(\frac{4}{3}\bigg)^s\Gamma\bigg(\frac{s}{3}\bigg)^{-1}(m-n)^{s/3-1}  &-\Gamma\bigg(\frac{4}{3}\bigg)^{s+1}\Gamma\bigg(\frac{s+1}{3}\bigg)^{-1}m^{(s+1)/3-1}\bigg| \notag \\
&\leqslant 4 \cdot \Gamma\bigg(\frac{4}{3}\bigg)^s\Gamma\bigg(\frac{s}{3}\bigg)^{-1}m^{s/3-1}.
\end{align}
Therefore combining \eqref{J1 bound Step 1} and \eqref{J1 bound Step 2}, and applying Lemma \ref{gamma approximation} we obtain
\begin{align*}
\bigg|J_1(m,s+1)&-\Gamma\bigg(\frac{4}{3}\bigg)^{s+1}\Gamma\bigg(\frac{s+1}{3}\bigg)^{-1}m^{(s+1)/3-1}\bigg|\\
&\leqslant \frac{10^{s-2}}{3}\sum_{n=1}^{m-1}n^{-2/3}(m-n)^{(s-1)/3-1}+4 \cdot \Gamma\bigg(\frac{4}{3}\bigg)^s\Gamma\bigg(\frac{s}{3}\bigg)^{-1}m^{s/3-1}\\
&\leqslant \frac{10^{s-2}}{3}m^{s/3-1}\Gamma\bigg(\frac{1}{3}\bigg)\Gamma\bigg(\frac{s-1}{3}\bigg)\Gamma\bigg(\frac{s}{3}\bigg)^{-1}+4\cdot10^{s-2}m^{(s-1)/3-1}\\
&\quad +4 \cdot \Gamma\bigg(\frac{4}{3}\bigg)^s\Gamma\bigg(\frac{s}{3}\bigg)^{-1}m^{s/3-1}
\leqslant 10^{s-1}m^{s/3-1}.
\end{align*}
By induction, the result follows.
\end{proof}
\section{Asymptotic Results for Representations as Sums of Icosahedral Numbers}\label{sec: Proof of Main Theorem}
Our goal in this section is to prove Theorem \ref{Theorem: Representation Icosahedral}. In fact, we will establish more general versions of Theorem \ref{Theorem: Representation Icosahedral}. Our first result is as follows.
\begin{thm}\label{main Theorem for formula}
For $m,s \in \mathbb{N}$, let $\mathcal{I}_s(m)$ denote the number of ways of representing $m$ as the sum of $s$ icosahedral numbers. Then for any $s\geqslant 9$, $0<\delta<\frac{1}{5}$, and $m>\frac{5}{2}(e^{e^{45}})^{3/\delta}$,
\begin{align}
\bigg|\mathcal{I}_s(m)& -\bigg(\frac{2}{5}\bigg)^{\frac{s}{3}}\mathfrak{S}(m)\Gamma\bigg(\frac{4}{3}\bigg)^{s}\Gamma\bigg(\frac{s}{3}\bigg)^{-1}m^{\frac{s}{3}-1}\bigg| \notag \\
     &\leqslant 113s\left(\frac{m}{2}\right)^{\frac{5\delta+s-4}{3}}+ \frac{(6\cdot 2^s+2\left(\frac{5}{4}\right)^{\frac{s}{3}}(s-3))\cdot 24^s\cdot\left(\frac{2}{5}\right)^{\frac{21s-(5s-42)\delta}{63}}}{(s-3)(\frac{5s}{21}-2)}m^{\frac{(42-5s)\delta+21s-63}{63}}\notag\\
     &\quad+182 s\left(\frac{1}{2}\right)^{\frac{s}{3}}e^{e^{46}} m^{\frac{s-4+2\delta}{3}}+\left(\frac{1}{5}\right)^{\frac{s}{3}}\left(\frac{1}{2}\right)^{\frac{3\delta-s\delta-9}{9}}\frac{6e^{e^{46}}2^s}{s-3}m^{\frac{3\delta-s\delta+3s-9}{9}}\notag\\
     &\quad+ \left( \frac{2}{5} \right)^{\frac{s}{3}} \cdot 10^{s-2} \cdot e^{e^{46}} m^{\frac{s-4}{3}}+152 \cdot 21^{s-8} \cdot m^{\frac{s}{3}-1-\frac{\delta(s-8)}{12}+\frac{0.53305(s-8)+6.3966}{1.2\log \log m}+\frac{(s-8)\log\log m}{4\log m}} \label{Icosahedral 1st Theorem Eq},
\end{align}
where $\mathfrak{S}(m)$ satisfies \eqref{Sigma Bound} and $ z=(2 e^{se^{89}}+1)^{\frac{21}{5s - 21}}.$
\end{thm}
\begin{proof} For $s>3, \lvert\theta\rvert\leqslant\frac{1}{2}$, we have $\lvert (\frac{5}{2})^{1/3}v(\theta)\rvert\leqslant\min\{\left(\frac{5}{2}\right)^{1/3}N,2\lvert\theta\rvert^{-1/3}\}$ following a similar computation as in Lemma \ref{V_1 theta Estimation}. Therefore, we obtain
\begin{align}
|J^*(m)|&= \left|\int_{-N^{\delta-3}}^{N^{\delta-3}}\left(\frac{5}{2}\right)^{s/3}v^s(\theta)e(-\theta m) \dd\theta\right| \leqslant \left(\frac{2^{s+1} \cdot 3}{s-3}+2\left(\frac{5}{4}\right)^{s/3}\right)m^{\frac{s}{3}-1}\label{|J^*(m)|}.
\end{align}
Now, we show that in \eqref{Approximating Major Arc Integral}, one may replace $J^{*}(m)$ by $J_1(m)$ while allowing a small error. Indeed, combining Lemma \ref{Singular Series Extension}, \eqref{J_1(m)-J^*(m)} and \eqref{|J^*(m)|}, we deduce that
\begin{align}
\bigg|\mathcal{I}_{s}^*(m)&- \bigg(\frac{2}{5}\bigg)^{\frac{s}{3}}\mathfrak{S}(m)J_1(m)\bigg| \notag \\
    &\leqslant \bigg(\frac{2}{5}\bigg)^{\frac{s}{3}}\bigg(|J^*(m)|\cdot |\mathfrak{S}(m)-\mathfrak{S}(m, N^\delta)|+|\mathfrak{S}(m)||J_1(m)-J^*(m)|\bigg)\notag\\
    &\leqslant\bigg(\frac{2}{5}\bigg)^{\frac{s}{3}} \bigg(\frac{(6\cdot 2^s+2\left(\frac{5}{4}\right)^{s/3}(s-3))\cdot 24^sm^{s/3-1}}{(s-3)(\frac{5s}{21}-2)N^{(\frac{5s}{21}-2)\delta}}+\bigg(72se^{e^{46}}\left(\frac{5}{2}\right)^{\frac{s}{3}}N^{s-4+2\delta}\notag\\  &\quad+6e^{e^{46}}\frac{2^{s}}{(s-3)}N^{\frac{3\delta-s\delta+3s-9}{3}}\bigg)\bigg)\notag\\
    &\leqslant \frac{(6\cdot 2^s+2\left(\frac{5}{4}\right)^{\frac{s}{3}}(s-3))\cdot 24^s\cdot(\frac{2}{5})^{\frac{21s-(5s-42)\delta}{63}}}{(s-3)(\frac{5s}{21}-2)}\cdot m^{\frac{(42-5s)\delta+21s-63}{63}} \notag\\
    &\quad+182s\left(\frac{1}{2}\right)^{\frac{s}{3}}e^{e^{46}} m^{\frac{s-4+2\delta}{3}}+\left(\frac{1}{5}\right)^{\frac{s}{3}}\left(\frac{1}{2}\right)^{\frac{3\delta-s\delta-9}{9}}\cdot \frac{6e^{e^{46}}2^s}{s-3}\cdot m^{\frac{3\delta-s\delta+3s-9}{9}}.\label{R*(m)-(2/5)^{s/3}S(m)J_1(m)}
\end{align}
Combining Lemma \ref{Approx 4} and \eqref{R*(m)-(2/5)^{s/3}S(m)J_1(m)}, we have
\begin{align}
\bigg|\int_{\mathfrak{M}}&f(\alpha)^s e(-\alpha m) \dd\alpha-\bigg(\frac{2}{5}\bigg)^{\frac{s}{3}}\mathfrak{S}(m)J_1(m)\bigg|\notag\\
&\leqslant
113s\left(\frac{m}{2}\right)^{\frac{5\delta+s-4}{3}}+ \frac{(6\cdot 2^s+2\left(\frac{5}{4}\right)^{s/3}(s-3))\cdot 24^s\cdot (\frac{2}{5})^{\frac{21s-(5s-42)\delta}{63}}}{(s-3)(\frac{5s}{21}-2)} \cdot m^{\frac{(42-5s)\delta+21s-63}{63}}\notag\\
&\quad +182s\left(\frac{1}{2}\right)^{\frac{s}{3}}e^{e^{46}} m^{\frac{s-4+2\delta}{3}}+\left(\frac{1}{5}\right)^{\frac{s}{3}}\left(\frac{1}{2}\right)^{\frac{3\delta-s\delta-9}{9}}\cdot \frac{6e^{e^{46}}2^s}{s-3} \cdot m^{\frac{3\delta-s\delta+3s-9}{9}}.
\label{f(alpha)-S(m)J_1(m) bound}
\end{align}
Finally, putting together Lemma \ref{J_1(m) approximation bound} and \eqref{f(alpha)-S(m)J_1(m) bound}, for $s\geqslant 9$ and $m>\frac{5}{2}(e^{e^{45}})^{3/\delta}$,
\begin{align}
\bigg|\int_{\mathfrak{M}}&f(\alpha)^s e(-\alpha m) \dd\alpha-\bigg(\frac{2}{5}\bigg)^{\frac{s}{3}}\Gamma\bigg(\frac{4}{3}\bigg)^{s}\Gamma\bigg(\frac{s}{3}\bigg)^{-1}\mathfrak{S}(m) m^{\frac{s}{3}-1}\bigg|\notag\\
 &\leqslant
113s\left(\frac{m}{2}\right)^{\frac{5\delta+s-4}{3}}+ \frac{(6\cdot 2^s+2\left(\frac{5}{4}\right)^{s/3}(s-3))\cdot 24^s\cdot (\frac{2}{5})^{\frac{21s-(5s-42)\delta}{63}}}{(s-3)(\frac{5s}{21}-2)}\cdot m^{\frac{(42-5s)\delta+21s-63}{63}}\notag\\
&\quad +182s\left(\frac{1}{2}\right)^{\frac{s}{3}}e^{e^{46}}m^{\frac{s-4+2\delta}{3}}+\left(\frac{1}{5}\right)^{\frac{s}{3}}\left(\frac{1}{2}\right)^{\frac{3\delta-s\delta-9}{9}}\frac{6e^{e^{46}}2^s}{s-3} \cdot m^{\frac{3\delta-s\delta+3s-9}{9}} \notag\\
&\quad + \left( \frac{2}{5} \right)^{\frac{s}{3}} \left( e^{e^{46}} \cdot 10^{s-2} m^{\frac{s-4}{3}}\right). \label{4.26 bounds}
\end{align}
The proof now follows immediately from Lemma \ref{Minor Arcs : thm} and \eqref{4.26 bounds}.
\end{proof}
In Theorem \ref{main Theorem for formula}, we can in fact choose a suitable $\delta$ in terms of $s$ to minimize the exponents of $m$ in the right-hand side of \eqref{Icosahedral 1st Theorem Eq}. Observing the first and second terms in the right-hand side, we must have $0<\delta<\frac{1}{5}$ so that the main term dominates the error term. In fact, any $\delta$ in this range satisfies the initial requirement in \eqref{Defining N}. Now, we optimize our choice of $\delta$.
\begin{thm}\label{main theorem asymp} 
One may choose $\delta = \frac{21}{63+5s}$ in the statement of Theorem \ref{main Theorem for formula}. 
\end{thm}
\begin{proof}
Extract the exponents depending on $\delta$ in the right-hand side of \eqref{Icosahedral 1st Theorem Eq}. Consider 
\[
G(\delta) = \max \bigg \{ 5\delta - 1, \bigg( 2 - \frac{5s}{21} \bigg )\delta, 2\delta-1, \delta\bigg (1 - \frac{s}{3}\bigg ), -\delta\bigg(\frac{s}{4} -2\bigg) \bigg \}.
\]
Let $\delta_{0}$ be such that $G(\delta_{0}) = \inf \{G(\delta): \delta \in (0,\frac{1}{5})\}$.
When $s\geqslant 9$, we have 
\[ 
G(\delta) = 
\begin{cases}
    \delta(1-\frac{s}{3}), \text{ if } \delta < 0 \\ \left( 2 - \frac{5s}{21} \right)\delta, \text{ if } \delta \in [0, \delta_{0}] \\ 5\delta - 1, \text{ if } \delta > \delta_{0}.
\end{cases}
\]
Thus $\delta_{0}$ occurs at the intersection of $\left( 2 - \frac{5s}{21} \right)\delta$ and $5\delta - 1$, which implies $\delta_0 =\frac{21}{63+5s}$.
\end{proof}
In particular, if we let $s=9$ and define 
\begin{align}\label{Arithmetic Factor for s=9}
\mathfrak{S}_{9,\mathcal{I}}(m) : = \sum_{q=1}^{\infty} \sum\limits_{\substack{a=1 \\ (a,q)=1}}^q\bigg(\frac{V(q,a)}{2q}\bigg)^9 e\bigg(-\frac{am}{q}\bigg),
\end{align}
where $V(q,a)$ and $I_n$ are given by \eqref{Definition V(q,a)} and\eqref{Icosahedral Number : defn} respectively, we obtain Theorem \ref{Theorem: Representation Icosahedral} as stated in Section \ref{sec: Introduction}. Even though Theorems \ref{main Theorem for formula} and \ref{main theorem asymp} hold for a much lower threshold for $m$, we need $m>e^{e^{94}}$ to ensure that in Theorem \ref{Theorem: Representation Icosahedral}, the error term is smaller than the main term.
\section{An Algebraic Approach towards Pollock's Conjectures} \label{sec: Proof of Theorem 1.6}
\subsection{Sums of Eight Icosahedral Numbers} Here we work with eight icosahedral numbers, hoping to get a much lower bound for $m$ compared to what we obtained in Theorem \ref{Theorem: Representation Icosahedral}. We will begin with Linnik's method and apply new algebraic ingredients that help us in completely proving Pollock's conjectures. To begin with, we first consider the sum of six icosahedral numbers. Let $g(x) = \frac{5x^3}{2}-\frac{5x^2}{2}+x,$ and write
\begin{align}\label{Six Icosahedral Numbers}
\mathcal{S}(6,\mathcal{I}) &= g(x + a + 1) +g(x - a) + g(x+b+1) + g(x- b) + g(x + c + 1) + g(x - c)\notag
\\
&= \frac{3}{2}(10x^{3} + 5x^{2} + 9x + 2) + (30x + 5)\left( \frac{c(c+1)}{2} + \frac{b(b+1)}{2} + \frac{a(a+1)}{2} \right).
\end{align}
Here $a,b,c <x$. Since $\frac{n(n+1)}{2}$ is the general form of the $n$-th triangular number, we see that 
\[
 \frac{c(c+1)}{2} + \frac{b(b+1)}{2} + \frac{a(a+1)}{2} 
\]
is the sum of three triangular numbers. Due to Gauss, we know that every positive integer is the sum of at most three triangular numbers. We use this to prove for any large and fixed $m$, 
\begin{align}\label{Eight Icosahedral Numbers}
\mathcal{S}(6,\mathcal{I}) + g(\ell) + g(r)= m,
\end{align}
where we will later choose $\ell$ and $r$ to be suitable positive integers. \par
From \eqref{Six Icosahedral Numbers} and \eqref{Eight Icosahedral Numbers}, it follows that 
\[
\left( \frac{c(c+1)}{2} + \frac{b(b+1)}{2} + \frac{a(a+1)}{2} \right) = \frac{m - g(\ell) - g(r) - \frac{3}{2}(10x^{3} + 5x^{2} + 9x + 2)}{30x + 5}.
\]
This implies that the expression
\[
m - g(\ell) - g(r) - \frac{3}{2}(10x^{3} + 5x^{2} + 9x + 2)
\]
must be divisible by $30x + 5$. Therefore, the problem is reduced as follows. For any positive integer $m$ (maybe sufficiently large), it suffices to find $x,\ell,r$ in $\mathbb{Z}^+$ satisfying:
\begin{align*}
   &1. \quad 0< L - g(\ell) - g(r) < 45x^3\\
   &2. \quad (30x+5)\mid (L - g(\ell) - g(r))
\end{align*}
where $L = m - \frac{3}{2} \left( 10x^{3} + 5x^{2} + 9x + 2 \right) \in \Z^+$. We will refer to the above conditions as Condition 1 and Condition 2, respectively. We treat them in the following subsection.
\subsection{The Arithmetic Conditions} Fix $t,\varepsilon \in \R^{+}$. Our initial strategy is to find a prime $p$ such that $p \equiv 1 \bmod 6$ and 
\begin{align}
    (t - \varepsilon)p^{3} \leqslant m \leqslant (t + \varepsilon)p^{3}.\label{m bound}
\end{align}
This means that $p$ needs to satisfy the inequalities
\begin{align}\label{Primes in AP Condition}
\bigg ( \frac{m}{t + \varepsilon} \bigg)^{1/3} \leqslant p \leqslant \bigg ( \frac{m}{t - \varepsilon} \bigg)^{1/3}.
\end{align}
For large enough $m$, by the Prime Number Theorem, there exists a prime in such an interval. Note that all primes except 2 and 3 are of the form $6n+1$ and $6n-1$. By the Prime Number Theorem for arithmetic progressions, these primes are almost evenly distributed when the range is large. Therefore, when $m$ is sufficiently large, the number of primes of the form $6n+1$ satisfying \eqref{Primes in AP Condition} is about 
\begin{align*}
    \frac{3\left(\frac{m}{t - \varepsilon} \right)^{1/3}}{2\log(\frac{m}{t - \varepsilon})}-\frac{3\left(\frac{m}{t + \varepsilon} \right)^{1/3}}{2\log(\frac{m}{t + \varepsilon})}.
\end{align*}
If such a $p$ is guaranteed, then we let $x = \frac{p-1}{6}$. Therefore, we have that $30x + 5 = 5p.$ Now we need to find suitable values of $\ell,r$ satisfying Conditions 1 and 2 with this chosen $x$. Substituting $x = \frac{p-1}{6}$ into \eqref{m bound}, we get
\begin{align*}
    &(t-\varepsilon)p^3 \leqslant L +\frac{5(p-1)^3}{72}+\frac{15(p-1)^2}{72}+\frac{9(p-1)}{4}+3\leqslant (t+\varepsilon)p^3.
\end{align*}
When $p$ is sufficiently large, we may reduce this to
\begin{align}\label{L Inequality}
\bigg (t-2\varepsilon-\frac{5}{72}\bigg )p^3\leqslant L\leqslant \bigg (t+\varepsilon-\frac{5}{72}\bigg)p^3.
\end{align}
We first consider Condition 1. This will be satisfied if
\begin{align}
L - \frac{5}{24}p^{3} < g(\ell) + g(r) < L.\label{alternative form of condition 1}  
\end{align}
Recalling the range of $L$, we must choose $r$ and $\ell$ such that
\begin{align}
   \left( t - 2\varepsilon - \frac{5}{18} \right) p^{3} < g(\ell) + g(r) < \left( t + \varepsilon - \frac{5}{72}\right) p^{3}. \label{f(l)+f(r) bound 2}
\end{align}
If $r, \ell \in \left( (\alpha - \varepsilon) p, (\alpha + \varepsilon) p\right)$ for some $\alpha>0$, then
\begin{align}
  g(\ell) + g(r) < 5(\alpha + \varepsilon)^{3}p^{3}.  \label{f(l)+f(r) upper bound}
\end{align}
Combining this with \eqref{f(l)+f(r) bound 2}, we may assume 
\begin{align*}
(\alpha + \varepsilon)^{3} < \frac{t}{5} - \frac{1}{72} + \frac{\varepsilon}{5},
\end{align*}
which implies that
\begin{align}\label{Alpha Range Lower}
\alpha < \bigg(\frac{t}{5} - \frac{1}{72} + \frac{\varepsilon}{5}\bigg)^{1/3}-\varepsilon.
\end{align}
Similarly, we need 
\begin{align}\label{flfr lower bound}
    5(\alpha - \varepsilon)^{3}p^{3}<g(\ell) + g(r),
\end{align}
so we may assume that
\begin{align}\label{Alpha Range Upper}
\alpha > \bigg(\frac{t}{5} - \frac{2\varepsilon}{5} - \frac{1}{18}\bigg)^{1/3}+\varepsilon.
\end{align}
Now we study Condition 2. To this end, we wish to find $\ell, r$ that satisfy the congruence $g(\ell) + g(r) \equiv L \bmod 5p.$ Using the Chinese Remainder Theorem, it suffices to consider the congruence equations $\ell+r \equiv L \bmod 5$ and $g(\ell) + g(r) \equiv L \bmod p$. Therefore, in order to satisfy both conditions, it suffices to find $\ell,r\in\mathbb{Z^+}$ satisfying
\begin{align}
r,\ell &\in \left( (\alpha - \varepsilon) p,(\alpha + \varepsilon) p\right),\notag\\
\ell &\equiv 0\bmod 5,\notag \\
r&\equiv L\bmod 5,\notag \\
\textrm{and} \quad  g(\ell)+g(r)&\equiv L\bmod p,\label{final conditions}
\end{align}
where $\alpha$ satisfies \eqref{Alpha Range Lower} and \eqref{Alpha Range Upper}. To address this problem, we will prove an explicit version of a result involving Lehmer points due to Cobeli and the fourth author \cite[Theorem 1]{cobeli}. We first introduce some notation following \cite{cobeli}. Let $1 \leqslant L_{0} \leqslant 5$ be such that $L \equiv L_{0} \bmod 5$. Let $\mathcal{C}$ denote the algebraic curve over $\mathbb{F}_{p}$ given by 
\begin{align}\label{Curve definition}
\mathcal{C}: g(x_1)+g(x_2)-L=0.
\end{align}
Define the vectors
\begin{align}
\textbf{x}& := \begin{bmatrix}
           x_1 \\
           x_2 \\
         \end{bmatrix},
    \textbf{a} := \begin{bmatrix}
           5 \\
           5 \\
         \end{bmatrix},
    \textbf{b} := \begin{bmatrix}
           0 \\
           L_0 \\
         \end{bmatrix},
    \textbf{t}:= \begin{bmatrix}
         t_1 \\
          t_2 \\
         \end{bmatrix},
         \textbf{g}:= \begin{bmatrix}
         g_1 \\
          g_2 \\
         \end{bmatrix}. \label{vectors}
\end{align}
The set of Lehmer points $\mathcal{L}(p,\mathcal{C},\textbf{a},\textbf{b})$ with respect to the prime $p$, the curve $\mathcal{C}$, and the vectors $\textbf{a},\textbf{b}$ is defined as
\begin{align}\label{Lehmer Point Definition}
\mathcal{L}(p,\mathcal{C},\textbf{a},\textbf{b}) := \{\textbf{x} : x_1 \equiv 0\bmod 5, x_2\equiv L_0 \bmod 5, (x_1,x_2)\in\mathcal{C}, 0\leqslant x_1, x_2<p\} . 
\end{align}
For the prime $p$ and the vectors $\textbf{t}$ and $\textbf{g}$, consider the intervals $U_j = ((t_j-g_j)p, (t_j+g_j)p)$ for $j \in \{1,2\}$. Then the distribution function of the Lehmer points is given by
\begin{align}\label{Lehmer Distribution function}
F(p,\mathcal{C},\textbf{a},\textbf{b};\textbf{g},\textbf{t}):=\#\{\textbf{x}\in \mathcal{L}(p,\mathcal{C},\textbf{a},\textbf{b}):  x_j\in U_j, 0\leqslant t_j\pm g_j\leqslant 1, j=1,2 \}.
\end{align}
With the above setup, we have the following lemma.
\begin{lem} \label{ZC theorem}
Let $p$ be a prime. Suppose $\mathcal{C}$, $\textnormal{\textbf{a}, \textbf{b}, \textbf{g}, \textbf{t}}$ and $ F$ be as in \eqref{Lehmer Distribution function}. Then for $p\geqslant 10^{10}$,
\begin{align*}
F(p,\mathcal{C},\textnormal{\textbf{a}, \textbf{b}; \textbf{g}, \textbf{t}}) = \frac{4g_1g_2}{25}(p+1)+ \mathcal{E},
\end{align*}
where $|\mathcal{E}| \leqslant 6.0003\sqrt{p}\log^2 p$.
\end{lem}
\begin{rem}
If $\mathcal{C}$ is reducible over $\mathbb{F}_p$ as a union of a line and a conic, then take $\mathcal{C}$ in the statement of Lemma \ref{ZC theorem} to be the conic part. Note that by the discussion in Lemma \ref{elliptic curve}, $\mathcal{C}$ can never be factored as a product of linear factors. 
\end{rem}
\begin{proof}
If $\mathcal{Q}$ is a subset of $\mathbb{F}_p^2$, we denote by $\chi_\mathcal{Q}$ the characteristic function of $\mathcal{Q}$, that is 
\begin{align*}
 \chi_\mathcal{Q}(\textbf{x}) =
    \begin{cases}
    1, & \text{if } \textbf{x} \in \mathcal{Q}\\
    0, & \text{otherwise.}
    \end{cases}
\end{align*}
For $j=1,2$, let $\mathcal{Q}(a_j,b_j,g_j,t_j)$ be the image in $\mathbb{F}_p$ of the set $ \{(t_j-g_j)p\leqslant x_j\leqslant (t_j+g_j)p: x_j\equiv b_j (\bmod \hspace{0.1cm} a_j),j=1,2\}.$ Then we have
\begin{align}
F(p,\mathcal{C},\textbf{a},\textbf{b};\textbf{g},\textbf{t}) &= \sum_{\textbf{x} \in\mathcal{C}}\chi_{\mathcal{Q}(a_1,b_1,g_1,t_1)}(x_1)\chi_{\mathcal{Q}(a_2,b_2,g_2,t_2)}(x_2)\label{F(p,r,C,a,b,t) Eq 1}.
\end{align}
Given $z \in \mathbb{F}_{p}$, let
\begin{align*}
\hat{\chi}_{\mathcal{Q}(a_j,b_j,g_j,t_j)}(z) &= \frac{1}{p}\sum_{w\in \mathbb{F}_p}\chi_{\mathcal{Q}(a_j,b_j,g_j,t_j)}(w)e_p(-wz),
\end{align*}
where $e_p$ is defined by $e_p(t) = e^{\frac{2 \pi i t}{p}}$. Then we can rewrite \eqref{F(p,r,C,a,b,t) Eq 1} as
\begin{align}
F(p,\mathcal{C},\textbf{a},\textbf{b};\textbf{g},\textbf{t}) =\sum_{\textbf{z}\in \mathbb{F}^2_p}\hat{\chi}_{\mathcal{Q}(a_1,b_1,g_1,t_1)}(z_1)\hat{\chi}_{\mathcal{Q}(a_2,b_2,g_2,t_2)}(z_2)S(\mathbf{z},p,\mathcal{C})\label{F(p,r,C,a,b,t)},
\end{align}
where $S(\mathbf{z},p,\mathcal{C}) = \sum_{\textbf{x} \in\mathcal{C}}e_p(z_1x_1+z_2x_2).$
Based on \eqref{vectors}, we choose $a_1=a_2=5$, $b_1=0$ and $b_2=L_0$. Let $m_j = [2g_j+2g_j [
       \frac{p-b_j}{5}]]$ for $j=1,2$. Then $\hat{\chi}_{\mathcal{Q}(a_j,b_j,g_j,t_j)}(0)=\frac{m_j}{p}.$
Therefore, it follows that \begin{align*}
 \bigg|\hat{\chi}_{\mathcal{Q}(a_j,b_j,g_j,t_j)}(0)-\frac{2g_j}{5}\bigg|\leqslant \frac{(4+2b_j)g_j}{p}.
       \end{align*}
When $z\neq 0$, we deduce that
\begin{align*}
\hat{\chi}_{\mathcal{Q}(a_j,b_j,q_j,t_j)}(z)&=\frac{e_p(-b_jz)-e_p(-(a_jm_j+b_j)z)}{p(1-e_p(-a_jz))} \leqslant \frac{2}{2p\left|\sin{\frac{\pi a_jz}{p}}\right|}.
\end{align*}
For $j=1,2$, we make a change of variables by choosing $u_j$ such that $0< \lvert u_j \rvert \leqslant \lvert \frac{p-1}{2} \rvert $ and $u_j \equiv a_j z \equiv5z\bmod p$. Then $p|\sin{\frac{\pi a_jz}{p}}|\geqslant 2|u_j|.$ In conclusion, for any $z\in \mathbb{F}_p$, we have
    \begin{align}
         \hat{\chi}_{\mathcal{Q}(a_j,b_j,g_j,t_j)}(z)\leqslant \min \bigg\{\frac{m_j}{p},\frac{1}{2|u_j|}\bigg\}\label{chi_{M(a_j,b_j,t_j)}(z)}.
    \end{align}
For $z_1=z_2=0$, we have $S(\textbf{0},p,\mathcal{C})=|\mathcal{C}|$, where $|\mathcal{C}|$ denotes the number of $\mathbb{F}_p$-rational points on the curve $\mathcal{C}$. By the previous remark and Lemma \ref{elliptic curve satisfied}, we have $|\mathcal{C}|=p+1+E$, where $|E|\leqslant 2\sqrt{p}$. In fact, if $\mathcal{C}$ is a conic, then $E=0$. So the contribution from $\mathbf{z}=\mathbf{0}$ in \eqref{F(p,r,C,a,b,t)} is
\begin{align*}
|\mathcal{C}|\hat{\chi}_{\mathcal{Q}(a_1,b_1,g_1,t_1)}(0)\hat{\chi}_{\mathcal{Q}(a_2,b_2,g_2,t_2)}(0)=(p+1+E)\left(\frac{2g_1}{5}+E_1\right)\left(\frac{2g_2}{5}+E_2\right),
\end{align*}
where $|E_1|\leqslant \frac{4g_1}{p}$ and $|E_2|\leqslant \frac{(4+2L_0)g_2}{p}$.
Therefore, since $0\leqslant g_j\leqslant 1/2$
and $L_0\leqslant 5$, we have
\begin{align}
\bigg||\mathcal{C}| &\hat{\chi}_{\mathcal{Q}(a_1,b_1,g_1,t_1)}(0)\hat{\chi}_{\mathcal{Q}(a_2,b_2,g_2,t_2)}(0)-\frac{4(p+1)g_1g_2}{25}\bigg|\notag\\
  &\leqslant \frac{(4(2+L_0)+8)g_1g_2}{5} \cdot \frac{p+1}{p}+8(2+L_0)g_1g_2\frac{p+1}{p^2}+\frac{8g_1g_2}{25}\sqrt{p}\notag\\
  &\quad+\frac{8(2+L_0)g_1g_2}{5\sqrt{p}}+\frac{16g_1g_2}{5\sqrt{p}}+\frac{18(2+L_0)g_1g_2}{p^{3/2}}\notag\\
  &\leqslant \frac{9}{5}\frac{p+1}{p}+14\frac{p+1}{p^2}+\frac{2}{25}\sqrt{p}+\frac{13}{5\sqrt{p}}+\frac{32}{p^{3/2}}\leqslant \frac{2}{25}\sqrt{p}+2,\label{Z-C theorem bound 1}
\end{align}
when $p\geqslant 10^{10}$. If $\mathbf{z}\neq \mathbf{0}$, then by our assumptions on the curve $\mathcal{C}$, we know that the linear form $z_1x_1+z_2x_2$ is not constant along $\mathcal{C}$. We apply the Bombieri--Weil inequality (see \cite{bombieri}, Theorem 6) by choosing $d_1=3$ and $d_2=1$. (Again, when $\mathcal{C}$ is a conic, $d_1 = 2$ and $d_2=1$.) It follows that $|S(\mathbf{z},p,\mathcal{C})|\leqslant 6\sqrt{p}+9\leqslant 6.0001\sqrt{p}.$ Therefore, the contribution from $\mathbf{z}\neq \mathbf{0}$ in \eqref{F(p,r,C,a,b,t)} is
\begin{align*}
    \leqslant 6.0001\sqrt{p}\sum_{\textbf{z}\in \mathbb{F}^2_p}\left|\hat{\chi}_{\mathcal{Q}(a_1,b_1,g_1,t_1)}(z_1)\right|\cdot\left|\hat{\chi}_{\mathcal{Q}(a_2,b_2,g_2,t_2)}(z_2)\right|.
\end{align*}
Using $\eqref{chi_{M(a_j,b_j,t_j)}(z)}$, this is bounded by 
\begin{align}
    &6.0001\sqrt{p}\sum_{z_1\in \mathbb{F}_p}\min \bigg\{\frac{2g_1+2g_1\left[
       \frac{p}{5}\right]}{p},\frac{1}{2|u_1|}\bigg\}\sum_{z_2\in \mathbb{F}_p}\min \bigg\{\frac{2g_2+2g_2\left[
       \frac{p-L}{5}\right]}{p},\frac{1}{2|u_2|}\bigg\}\notag\\
       &\leqslant6.0001\sqrt{p}\sum_{z_1\in \mathbb{F}_p}\min \bigg\{\frac{2}{p}+\frac{1}{5},\frac{1}{2|u_1|}\bigg\}\sum_{z_2\in \mathbb{F}_p}\min \bigg\{\frac{2}{p}+\frac{1}{5},\frac{1}{2|u_2|}\bigg\}.\label{Z-C theorem bound 2}
\end{align}
As $z_1$ varies in $\mathbb{F}_p$, $u_1=u_1(a_1,z_1)$ runs over the set $\{-\frac{p-1}{2},\ldots,\frac{p-1}{2}\}$ and the same holds true for $z_2$ and $u_2(a_2,z_2)$. Therefore \eqref{Z-C theorem bound 2} is bounded by
$6.0001\sqrt{p}\log^2 p$. Combining this with the bound in \eqref{Z-C theorem bound 1}, we have that the total error is bounded by $6.0003\sqrt{p}\log^2 p$.
\end{proof}
\subsection{Proof of Theorems \ref{Theorem: 8 Icosahedral Numbers} and \ref{Theorem: Pollock Icosahedral}} We now present the proof of Theorem \ref{Theorem: 8 Icosahedral Numbers}.
\begin{proof}
We will actually show that any  $m\geqslant \num{9.6446e35}$ is a sum of 8 icosahedral numbers. We choose $t=0.39,\varepsilon=0.16$ and $\alpha=0.16$ satisfying \eqref{Alpha Range Lower} and \eqref{Alpha Range Upper}. Since $(\alpha\pm\varepsilon)\in [0,1]$, we let $t_1=t_2=\alpha$ and $g_1=g_2=\varepsilon$ in Lemma \ref{ZC theorem}. Then the number of Lehmer points $(\ell,r)$ that satisfies the requirements in \eqref{final conditions} is
\[F(p,\mathcal{C},\textbf{a},\textbf{b};\textbf{g}, \textbf{t})=\frac{4\varepsilon^2(p+1)}{25}+\mathcal{E},
\]
where $|\mathcal{E}| <6.0003\sqrt{p}\log^2p$ by Lemma \ref{ZC theorem}. Hence to have $F(p,\mathcal{C},\textbf{a},\textbf{b};\textbf{g}, \textbf{t})>0$, we need 
\begin{align}\label{Prime Requirement}
\varepsilon^2(p+1)>\frac{25}{4} (6.0003\sqrt{p}\log^2p).
\end{align}
So when $m\geqslant \num{9.6446e35}$, by \eqref{m bound} $p\geqslant \num{1.299e12}$. Therefore, $0.16> 6.12388p^{-1/4} \log p$
which implies \eqref{Prime Requirement}.
So when $m\geqslant \num{9.6446e35}$, if we get a prime $p\equiv 1\bmod 6$ such that 
\[\frac{1}{\sqrt[3]{0.55}}\sqrt[3]{m}\leqslant p \leqslant \frac{1}{\sqrt[3]{0.23}}\sqrt[3]{m},
\]
then from the discussion above, the existence of $(\ell,r)$ satisfying \eqref{final conditions} is guaranteed, so Condition 2 will be met. Now, we restrict the range of $p$ to be the sub-interval [$1.54\sqrt[3]{m}, 1.62\sqrt[3]{m}$]. This range of $p$, when $p\geqslant 89$, together with \eqref{f(l)+f(r) upper bound} and \eqref{flfr lower bound} guarantees us \eqref{alternative form of condition 1}, which is exactly Condition 1. Thus the problem is reduced to finding a prime $p\equiv 1\bmod 6$ in the range [$1.54\sqrt[3]{m}, 1.62\sqrt[3]{m}$] when $m\geqslant \num{9.6446e35}$.
\par
Recall the Chebyshev function for arithmetic progressions, 
\begin{align*}
\theta(x;k,a)=\sum\limits_{\substack{p\leqslant x \\ p \equiv a \bmod k}}\log p, 
\end{align*}
where $p$ denotes a prime number. Note that $\theta(x;k,a)>0$ if and only if there exists a $p\leqslant x$ such that $p\equiv a\bmod k$. In our case, we take $k=6$ and $a=1$. By Ramaré and Rumely (see \cite[Theorem 1]{ramare1996}), for such $k$ and $a$, and for $m\geqslant \num{9.6446e35}$, we have
\begin{align*}
    \max_{1\leqslant y\leqslant 1.54\sqrt[3]{m}}\bigg|\theta(y;6,1)-\frac{y}{2}\bigg|\leqslant 0.0035112\sqrt[3]{m},
\end{align*}
where we take $\varepsilon$ in \cite{ramare1996} to be $0.00456$. Taking $y=1.54\sqrt[3]{m}$, we have
\begin{align}
    \bigg|\theta(1.54\sqrt[3]{m};6,1)-0.77\sqrt[3]{m}\bigg|\leqslant 0.0035112\sqrt[3]{m}.\label{upper left bound}
\end{align}
Similarly, 
\begin{align*}
    \max_{1\leqslant y\leqslant 1.62\sqrt[3]{m}}\bigg|\theta(y;6,1)-\frac{y}{2}\bigg|\leqslant 0.0036936\sqrt[3]{m}.
\end{align*}
Setting $y=1.62\sqrt[3]{m}$, we have
\begin{align}
   \bigg|\theta(1.62\sqrt[3]{m};6,1)-0.81\sqrt[3]{m}\bigg|\leqslant 0.0036936\sqrt[3]{m}.\label{lower right bound}
\end{align}
Putting together \eqref{upper left bound} and \eqref{lower right bound}, we see that 
\begin{align}
    \sum\limits_{\substack{1.54\sqrt[3]{m}<p\leqslant 1.62\sqrt[3]{m} \\ p \equiv 1 \bmod 6}}\log p&=\theta(1.62\sqrt[3]{m};6,1)-\theta(1.54\sqrt[3]{m};6,1)\notag\\
    &\geqslant (0.81\sqrt[3]{m}-0.0036936\sqrt[3]{m})-(0.77\sqrt[3]{m}+0.0035112\sqrt[3]{m})\notag\\
    & = 0.0327952\sqrt[3]{m} > 0.\label{lower bound > 0}
\end{align}
Therefore, there must exist a prime $p\equiv 1\bmod 6$ in the range $[1.54\sqrt[3]{m}, 1.62\sqrt[3]{m}]$. Thus we conclude that any $m\geqslant \num{9.6446e35}$, and subsequently any $m \geqslant 10^{36}$, can be written as a sum of 8 icosahedral numbers. 
\end{proof}
Theorem \ref{Theorem: Pollock Icosahedral} now follows from the bound for $m$ in Theorem \ref{Theorem: 8 Icosahedral Numbers} along with numerical verification, e.g., see details in \cite{github}.
\subsection{Proof of Theorems \ref{Theorem: 8 Dodecahedral Numbers} and \ref{Theorem: Pollock Dodecahedral}}
The proofs of Theorem \ref{Theorem: 8 Dodecahedral Numbers} and \ref{Theorem: Pollock Dodecahedral} follow similar arguments as in the icosahedral case. We allow ourselves to reuse notations from the icosahedral case. Let $g(x) = \frac{9x^3}{2}-\frac{9x^2}{2}+x$ and write
\begin{align}\label{Six dodecahedral Numbers}
\mathcal{S}({6,\mathcal{D}}) &= g(x + a + 1) +g(x - a) + g(x+b+1) + g(x- b) + g(x + c + 1) + g(x - c)\notag\\
&=\frac{3}{2}(18x^{3} + 9x^{2} + 13x + 2) + (54x + 9)\left( \frac{c(c+1)}{2} + \frac{b(b+1)}{2} + \frac{a(a+1)}{2} \right).
\end{align}
Here $a,b,c <x$. Again, we aim to show that for any large and fixed $m$, 
\begin{align}\label{Eight dodecahedral Numbers}
\mathcal{S}(6,\mathcal{D}) + g(\ell) + g(r)= m,
\end{align}
where we will later choose $\ell$ and $r$ to be suitable positive integers. Using Gauss' theorem on triangular numbers, our problem is reduced to finding $x,\ell,r\in\Z^+$ satisfying the conditions:
\begin{align*}
   &1 \quad 0< L - g(\ell) - g(r) < 81x^3,\\
   &2. \quad (54x+9)\mid (L - g(\ell) - g(r)),
\end{align*}
where $L = m - \frac{3}{2}(18x^{3} + 9x^{2} + 13x + 2) \in\Z^+$. We will refer to the above conditions as Condition 1 and Condition 2, respectively. 
\par
Fix $t,\varepsilon \in \R^{+}$. Our goal is to find a prime $p$ such that $p \equiv 1 \bmod 6$ and 
\begin{align}
    (t - \varepsilon)p^{3} \leqslant m \leqslant (t + \varepsilon)p^{3}.\label{m bound dode}
\end{align}
If such $p$ is guaranteed, then we let $x=\frac{p-1}{6}$. Condition 1 is satisfied if
\begin{align}
  L - \frac{3}{8}p^{3} < g(\ell) + g(r) < L.\label{alternative form of condition 1 dode}  
\end{align}
Let $\alpha\in\R^+$ such that $r, \ell \in \left( (\alpha - \varepsilon) p, (\alpha + \varepsilon) p\right)$. Similar to the icosahedral case, we have 
\begin{align}
\alpha\in\bigg( \left(\frac{t}{9} - \frac{2\varepsilon}{9} - \frac{1}{18}\right)^{\frac{1}{3}}+\varepsilon,\left(\frac{t}{9} - \frac{1}{72} + \frac{\varepsilon}{9}\right)^{\frac{1}{3}}-\varepsilon\bigg).\label{alpha condition dode}
\end{align}
Now we turn to Condition 2. To this end, it suffices to find $\ell,r\in\mathbb{Z^+}$ satisfying
\begin{align}
    r,\ell &\in \left( (\alpha - \varepsilon) p, (\alpha + \varepsilon) p\right),\notag\\
    \ell &\equiv0\bmod9,\notag\\
    r&\equiv L\bmod 9,\notag\\
    \textrm{and} \quad  g(\ell)+g(r)&\equiv L\bmod p,\label{final conditions dode}
\end{align} 
where $\alpha$ satisfies \eqref{alpha condition dode}.
We will again prove an explicit version of a result involving Lehmer points due to Cobeli and the fourth author \cite[Theorem 1]{cobeli} to tackle this problem. Our set-up is as follows. Let $1 \leqslant L_{0} \leqslant 9$ with $L \equiv L_{0} \bmod 9$. Let $\mathcal{C}$ be the algebraic curve over $\mathbb{F}_{p}$ given by 
\begin{align}\label{Curve definition dode}
\mathcal{C}: g(x_1)+g(x_2)-L=0.
\end{align}
Define the vectors
\begin{align}
    \textbf{x}& = \begin{bmatrix}
           x_1 \\
           x_2 \\
         \end{bmatrix},
    \textbf{a} = \begin{bmatrix}
           9 \\
           9 \\
         \end{bmatrix},
    \textbf{b} = \begin{bmatrix}
           0 \\
           L_0 \\
         \end{bmatrix},
    \textbf{t}= \begin{bmatrix}
         t_1 \\
          t_2 \\
         \end{bmatrix},
         \textbf{g}= \begin{bmatrix}
         g_1 \\
          g_2 \\
         \end{bmatrix}, \label{vectors dode}
\end{align}
The set of Lehmer points $\mathcal{L}(p,\mathcal{C},\textbf{a},\textbf{b})$ with respect to $p$, the curve $\mathcal{C}$, and the vectors $\textbf{a},\textbf{b}$ is
\begin{align}\label{Lehmer Point Definition dode}
\mathcal{L}(p,\mathcal{C},\textbf{a},\textbf{b}) &:= \#\{\textbf{x}: x_1 \equiv 0\bmod 9, x_2 \equiv L_0 \bmod 9, (x_1,x_2)\in\mathcal{C}, 0\leqslant x_1, x_2<p\} . 
\end{align}
For $p$ prime and the vectors $\textbf{t}$ and $\textbf{g}$, consider the intervals $U_j = ((t_j-g_j)p, (t_j+g_j)p)$ for $j \in \{1,2\}$. Then the distribution function of the Lehmer points is given by
\begin{align}\label{Lehmer Distribution function dode}
F(p,\mathcal{C},\textbf{a},\textbf{b};\textbf{g},\textbf{t}):=\#\{\textbf{x}\in \mathcal{L}(p,\mathcal{C},\textbf{a},\textbf{b}):  x_j\in U_j, 0\leqslant t_j\pm g_j\leqslant 1, j=1,2 \}.
\end{align}
With the above setup, we have the following lemma, which is analogous to Lemma \ref{ZC theorem}. We omit the proof here to avoid redundancy.
\begin{lem}\label{ZC theorem dode}
Let $p$ be a prime, and let $\mathcal{C}$, $\textnormal{\textbf{a}, \textbf{b}, \textbf{g}, \textbf{t}}$, and $ F$ be as in \eqref{Lehmer Distribution function dode}. For $p\geqslant 10^{10}$,
\begin{align*}
F(p,\mathcal{C},\textnormal{\textbf{a}, \textbf{b}; \textbf{g}, \textbf{t}}) = \frac{4g_1g_2}{81}(p+1)+ \mathcal{E},
\end{align*}
where $|\mathcal{E}| \leqslant 6.0003\sqrt{p}\log^2 p$.
\end{lem}
\begin{rem}
    Similar to the icosahedral case, if $\mathcal{C}$ is reducible over $\mathbb{F}_p$ as a union of a line and a conic, then take $\mathcal{C}$ in the statement of Lemma \ref{ZC theorem dode} to be the conic part. Again, $\mathcal{C}$ can never be factored as a product of linear factors. 
\end{rem}
We now present the proof of Theorem \ref{Theorem: 8 Dodecahedral Numbers}.
\begin{proof}
We will actually show that any $m\geqslant \num{5.04e38}$ is a sum of 8 dodecahedral numbers. We choose $t=0.32$ and $\varepsilon=0.17$. Choose $\alpha = 0.173$, which satisfies \eqref{alpha condition dode}. Using Lemma \ref{ZC theorem dode}, in order to have $F(p,\mathcal{C},\textbf{a},\textbf{b};\textbf{g}, \textbf{t})>0$, it suffices to have 
\begin{align}\label{Dode Requirement}
\varepsilon^2(p+1)>\frac{81}{4} (6.0003\sqrt{p}\log^2p).
\end{align}
When $m\geqslant \num{5.04e38}$, we have $p\geqslant \num{1.497e13}$ and therefore $0.17 > 11.023 p^{-1/4} \log p,$ which implies \eqref{Dode Requirement}. Therefore, the problem reduces to finding a prime $p\equiv 1\bmod 6$ such that 
\[\frac{1}{\sqrt[3]{0.49}}\sqrt[3]{m}\leqslant p \leqslant \frac{1}{\sqrt[3]{0.15}}\sqrt[3]{m}
\]
when $m\geqslant \num{5.04e38}$. We restrict the range of $p$ to be the sub-interval [$1.26\sqrt[3]{m}, 1.269\sqrt[3]{m}$]. This range of $p$, when $p\geqslant 272$, guarantees us Condition 1. Thus, the problem is further reduced to finding a prime $p\equiv 1\bmod 6$ in the range [$1.26\sqrt[3]{m}, 1.269\sqrt[3]{m}$] when $m\geqslant \num{5.04e38}$. Now we apply the result due to Ramaré and Rumely \cite[Theorem 1]{ramare1996} with $\varepsilon=0.002657$ to obtain 
\begin{align}
\sum\limits_{\substack{1.26\sqrt[3]{m}<p\leqslant 1.269\sqrt[3]{m} \\ p \equiv 1 \bmod 6}}\log p \geqslant 0.0011402235\sqrt[3]{m} > 0.\label{lower bound > 0 dode}
\end{align}
Therefore there must exist a prime $p\equiv 1\bmod 6$ in the range [$1.26\sqrt[3]{m}, 1.269\sqrt[3]{m}$]. Hence we conclude that any $m\geqslant \num{5.04e38}$, and subsequently any $m \geqslant 10^{39}$, can be written as a sum of 8 dodecahedral numbers. 
\end{proof}

The proof of Theorem \ref{Theorem: Pollock Dodecahedral} follows similarly to that of Theorem \ref{Theorem: Pollock Icosahedral}, see details in \cite{github}.

\section*{Acknowledgements}
The authors are very grateful to the anonymous referee for their valuable suggestions that have greatly increased the clarity and value of the manuscript. Debmalya Basak is partially supported by the Juliette Alexandra Trjitzinsky Fellowship and the Bateman Fellowship, Department of Mathematics, University of
Illinois Urbana-Champaign.

\printbibliography

@article{deshouillers20007373170279850,
  title={7373170279850},
  author={Deshouillers, J.-M. and Hennecart, F. and Landreau, B. and Purnaba, I. G. P.},
  journal={Math. Comp},
  volume={69},
  number={229},
  pages={421--439},
  year={2000}
}

@article{romani1982computations,
  title={Computations concerning Waring's problem for cubes},
  author={Romani, F.},
  journal={Calcolo},
  volume={19},
  number={4},
  pages={415--431},
  year={1982},
  publisher={Springer}
}

@article{elkies2010every,
  title={Every even number greater than 454 is the sum of seven cubes},
  author={Elkies, N. D.},
  journal={arXiv preprint arXiv:1009.3983},
  year={2010}
}

@article{ramare2005explicit,
  title={An explicit seven cube theorem},
  author={Ramaré, O.},
  journal={Acta Arith.},
  volume={118},
  number={4},
  pages={375},
  year={2005},
  publisher={The Institute of Mathematics}
}

@book{maillet1895decomposition,
  title={On the decomposition of an integer into a sum of cubes of positive integers},
  author={Maillet, E.},
  year={1895},
  publisher={To the Secretariat of the Association}
}

@article{mccurley1984effective,
  title={An effective seven cube theorem},
  author={McCurley, K. S.},
  journal={J. Number Theory},
  volume={19},
  number={2},
  pages={176--183},
  year={1984},
  publisher={Elsevier}
}

@article{boklan2009every,
  title={Every multiple of 4 except 212, 364, 420, and 428 is the sum of seven cubes},
  author={Boklan, K. D. and Elkies, N. D.},
  journal={arXiv preprint arXiv:0903.4503},
  year={2009}
}

@article{linnik1943representation,
  title={On the representation of large numbers as sums of seven cubes},
  author={Linnik, Y. V.},
  journal={Mat. Sb.},
  volume={12},
  number={2},
  pages={218--224},
  year={1943},
  publisher={Russian Academy of Sciences, Mathematical Institute named after. VA Steklov}
}

@article{pollock,
 author = {Pollock, S. F.},
 journal = {Proc. Roy. Soc. London},
 title = {On the extension of the principle of Fermat’s theorem of the polygonal numbers to the higher
orders of series whose ultimate differences are constant.},
 year = {1843-1850},
volume = {5},
pages = {922–924}

}

@article{hassecongruentfunctions,
author = {Hasse, H.},
    title = {Beweis des analogons der riemannschen vermutung fr die
artinschen und f. k. schmidtschen kongruenzzetafunktionen in gewissen
elliptischen fllen. vorlufige mitteilung},
    journal = {Vorl. Mitt. Nachr.Ges. Wiss. G\"{o}ttingen},
    pages = {253-262},
    volume = {42},
    year = {1933}
}

@article{hassecomplexmultiplication,
    author = {Hasse, H.},
    title = {Abstrakte begr¨undung der komplexen multiplikation und
    riemannsche vermutung in funktionenk¨orpern},
    journal = {Abh. Math. Sem. Univ.},
    pages = {325-348},
    volume = {10(1)},
    year = {1934}
}

@article{hassefiniteorder,
    author = {Hasse, H.},
    title = {Zur Theorie der abstrakten elliptischen Funktionenk¨orper
    I,II,III. Die Struktur der Gruppe der Divisorenklassen endlicher Ordnung},
    journal = {J. Reine Angew. Math.},
    pages = {55--62},
    volume = {175},
    year = {1936}
}

@article{bombieri,
 URL = {http://www.jstor.org/stable/2373048},
 author = {Bombieri, E.},
 journal = {Amer. J. Math.},
 number = {1},
 pages = {71--105},
 publisher = {Johns Hopkins University Press},
 title = {On Exponential Sums in Finite Fields},
 urldate = {2024-05-23},
 volume = {88},
 year = {1966}
}

@article{Brady_2015,
   title={Sums of seven octahedral numbers},
   volume={93},
   url={http://dx.doi.org/10.1112/jlms/jdv061},
   DOI={10.1112/jlms/jdv061},
   number={1},
   journal={J. Lond. Math. Soc. (2)},
   publisher={Wiley},
   author={Brady, Z. E.},
   year={2015},
   pages={244–272} }

@article{cobeli,
author = {Cobeli, C. and Zaharescu, A.},
year = {2001},
pages = {301-307},
title = {Generalization of a problem of Lehmer},
volume = {104},
journal = {Manuscripta Math.},
doi = {10.1007/s002290170028}
}

@article{langweil,
 URL = {http://www.jstor.org/stable/2372655},
 author = {Lang, S. and Weil, A.},
 journal = {Amer. J. Math.},
 number = {4},
 pages = {819--827},
 publisher = {Johns Hopkins University Press},
 title = {Number of Points of Varieties in Finite Fields},
 urldate = {2024-03-22},
 volume = {76},
 year = {1954}
}

@article{weil1949,
    author = {Weil, A.},
    title = {Numbers of solutions of equations in finite fields},
    journal = {Bull. Amer. Math. Soc.},
    year = {1949},
    doi = {https://doi.org/10.1090%2FS0002-9904-1949-09219-4},
    volume = {55},
    pages = {497–508}
}

@article{hua1936,
 URL = {http://www.jstor.org/stable/2370973},
 author = {Hua, L.-K. },
 journal = {Amer. J. Math.},
 number = {3},
 pages = {553--562},
 publisher = {Johns Hopkins University Press},
 title = {On Waring's Problem with Polynomial Summands},
 urldate = {2024-05-23},
 volume = {58},
 year = {1936}
}

@article{huagenwaring,
author = {Hua, L.-K. },
title = {On a Generalized Waring Problem},
journal = {Proc. Lond. Math. Soc. (3)},
volume = {s2-43},
number = {1},
pages = {161-182},
doi = {https://doi.org/10.1112/plms/s2-43.3.161},
url = {https://londmathsoc.onlinelibrary.wiley.com/doi/abs/10.1112/plms/s2-43.3.161},
year = {1937}
}

@article{hua1940,
author = {Hua, L.-K. },
    title = {On Waring’s problem with cubic polynomial summands},
    journal = {J. Indian Math. Soc. (N.S.)},
    pages = {127-135},
    volume = {4},
    year = {1940}
}

@article{Kempner1912BemerkungenZW,
  title={Bemerkungen zum Waringschen Problem},
  author={Kempner, A. J.},
  journal={Math. Ann.},
  year={1912},
  volume={72},
  pages={387-399},
  url={https://api.semanticscholar.org/CorpusID:120101223}
}

@book{vaughan1997hardy,
  title={The Hardy-Littlewood method},
  author={Vaughan, R. C.},
  year={1997},
  publisher={Cambridge University Press}
}

@article{ramar2007explicit,
  title={An explicit result of the sum of seven cubes},
  author={Ramar{\'e}, O.},
  journal={Manuscripta Math.},
  volume={124},
  number={1},
  pages={59--75},
  year={2007},
  publisher={Springer-Verlag, Berlin/Heidelberg}
}

@article {Hardy-Ramanujan,
    AUTHOR = {Hardy, G. H. and Ramanujan, S.},
     TITLE = {Asymptotic {F}ormulaae in {C}ombinatory {A}nalysis},
   JOURNAL = {Proc. London Math. Soc. (2)},
  FJOURNAL = {Proceedings of the London Mathematical Society. Second Series},
    VOLUME = {17},
      YEAR = {1918},
     PAGES = {75--115},
      ISSN = {0024-6115},
   MRCLASS = {99-04},
  MRNUMBER = {1575586},
       DOI = {10.1112/plms/s2-17.1.75},
       URL = {https://doi.org/10.1112/plms/s2-17.1.75},
}

@misc{bbz2025,
  author = {D. Basak and B. Berndt and A. Zaharescu},
  title = {Distributions Associated to Small Quadratic Non-Residues and their Partitions},
  year = {2025},
  note = {\href{https://drive.google.com/file/d/1YsZjnFmMTqQx5Gdr95CrzwiduqiSDA83/view?usp=sharing}{URL}, submitted}
}

@article {LargeSieveInequality,
    AUTHOR = {Prakash, G. and Ramana, D. S.},
     TITLE = {The large sieve inequality for integer polynomial amplitudes},
   JOURNAL = {J. Number Theory},
  FJOURNAL = {Journal of Number Theory},
    VOLUME = {129},
      YEAR = {2009},
    NUMBER = {2},
     PAGES = {428--433},
      ISSN = {0022-314X,1096-1658},
   MRCLASS = {11N36},
  MRNUMBER = {2473889},
MRREVIEWER = {Zaizhao\ Meng},
       DOI = {10.1016/j.jnt.2008.05.006},
       URL = {https://doi.org/10.1016/j.jnt.2008.05.006},
}

@article {Ramachandra,
AUTHOR = {Ramachandra, K. and Sankaranarayanan, A.},
TITLE = {On an asymptotic formula of {S}rinivasa {R}amanujan}, JOURNAL = {Acta Arith.}, 
FJOURNAL = {Acta Arithmetica}, 
VOLUME = {109}, 
YEAR = {2003}, 
NUMBER = {4}, 
PAGES = {349--357}, 
ISSN = {0065-1036,1730-6264}, 
MRCLASS = {11N37 (11M06)}, 
MRNUMBER = {2009049}, 
MRREVIEWER = {D.\ R.\ Heath-Brown}, 
DOI = {10.4064/aa109-4-5}, 
URL = {https://doi.org/10.4064/aa109-4-5}, 
}

@article{nicolas1983majorations,
  title={Majorations explicites pour le nombre de diviseurs de N},
  author={Nicolas, J.-L. and Robin, G.},
  journal={Canad. Math. Bull.},
  volume={26},
  number={4},
  pages={485--492},
  year={1983},
  publisher={Cambridge University Press}
}

@article{ramare1996,
  title={Primes in Arithmetic Progressions},
  author={Ramaré, O. and Rumely, R.},
  journal={Math. Comp},
  volume={65},
  number={213},
  pages={397--425},
  year={1996},
  publisher={American Mathematical Society}
}

@article{siksek2016every,
  title={Every integer greater than 454 is the sum of at most seven positive cubes},
  author={Siksek, S.},
  journal={Algebra Number Theory},
  volume={10},
  number={10},
  pages={2093--2119},
  year={2016}
}

@article{watsoncube,
author = {Watson, G. L.},
title = {A Proof of the Seven Cube Theorem},
journal = {J. Lond. Math. Soc. (2)},
volume = {s1-26},
number = {2},
pages = {153-156},
doi = {https://doi.org/10.1112/jlms/s1-26.2.153},
url = {https://londmathsoc.onlinelibrary.wiley.com/doi/abs/10.1112/jlms/s1-26.2.153},
year = {1951}
}

@article{Wieferich1908BeweisDS,
  title={Beweis des Satzes, da{\ss} sich eine jede ganze Zahl als Summe von h{\"o}chstens neun positiven Kuben darstellen l{\"a}{\ss}t},
  author={Wieferich, A.},
  journal={Math. Ann},
  year={1908},
  volume={66},
  pages={95-101},
  url={https://api.semanticscholar.org/CorpusID:121386035}
}

@misc{dusart2010estimates,
      title={Estimates of Some Functions Over Primes without R.H.}, 
      author={Dusart, P.},
      year={2010},

}

@online{github,
  author = {Dong, A.},
  title = {Algorithm for Pollock's Icosahedral and Dodecahedral Numbers Conjectures},
  year = {2024},
  publisher = {GitHub},
  note = {Available at: \href{https://github.com/Anjisweety2/Algorithm-for-Pollock-s-icosahedral-and-dodecahedral-numbers-conjectures}{Computations for Pollock's Conjectures}}
}
\end{document}